\documentclass[11pt, reqno, oneside, notitlepage]{amsart}

\usepackage[a4paper, total={6in, 9.1in}]{geometry}

\usepackage{amsmath,amscd}
\usepackage{amssymb}
\usepackage{amsthm}
\usepackage{comment}
\usepackage{graphicx}
\usepackage{epstopdf} 
\usepackage{mathrsfs}

\usepackage{bm}
\usepackage{hyperref}
\usepackage{xcolor}
\hypersetup{
	colorlinks,
	linkcolor={blue},
	urlcolor={blue},
	citecolor={red}
}

\theoremstyle{plain}
\newtheorem{thm}{Theorem}[section]
\newtheorem{prop}{Proposition}[section]

\newtheorem{lem}[prop]{Lemma}
\newtheorem{cor}[prop]{Corollary}

\newtheorem{defi}[prop]{Definition}
\newtheorem{rmk}[prop]{Remark}

\newtheorem*{proposition*}{Proposition}

\numberwithin{equation}{section}
\newcommand {\R} {\mathbb{R}} 
 \newcommand {\N} {\mathbb{N}}
\newcommand {\p} {\partial}

\newcommand{\eps}{\epsilon}

\newcommand {\supp} {\text{supp}}

\newcommand{\wt}{\widetilde}
\newcommand{\vphi}{\varphi}

\newcommand{\LC}{\left(}
\newcommand{\RC}{\right)}

\newcommand{\norm}[1]{\lVert #1 \rVert}         

\DeclareMathOperator {\dist} {dist}

\title[Determination of a nonlinear hyperbolic system]{Determining a nonlinear hyperbolic system with unknown sources and nonlinearity}

\author[Y.-H. Lin]{Yi-Hsuan Lin}
\address{Department of Applied Mathematics, National Yang Ming Chiao Tung University, Hsinchu, Taiwan}
\email{yihsuanlin3@gmail.com}

\author[H. Liu]{Hongyu Liu}
\address{Department of Mathematics, City University of Hong Kong, Kowloon, Hong Kong SAR, China}
\email{hongyu.liuip@gmail.com, hongyliu@cityu.edu.hk}

\author[X. Liu]{Xu Liu}
\address{Key Laboratory of Applied Statistics of MOE,  
	School of Mathematics and Statistics,  Northeast Normal University, Changchun, China}
\email{liux216@nenu.edu.cn}

\begin{document}
	\maketitle
	
	\begin{abstract}
		
		This paper is devoted to  some inverse boundary problems associated with a time-dependent semilinear hyperbolic  equation,  where both  nonlinearity and  sources (including  initial displacement and  initial velocity) are unknown.  It is shown in several generic scenarios that one can uniquely determine the nonlinearity and/or the sources by using passive or active boundary observations.  In order to exploit  the nonlinearity and the sources  simultaneously, we develop a new technique,  
		which combines the observability for linear wave  equations  and an approximation property with higher order linearization for the semilinear hyperbolic equation. \medskip


		
	\end{abstract}

	\tableofcontents

	\section{Introduction}\label{Sec 1}
	
	The inverse problems for nonlinear partial differential equations (PDEs) have received considerable attention in the literature. 
	In \cite{14}, an inverse boundary problem was proposed for a nonlinear parabolic PDE, and it was shown that the first linearization of the boundary Dirichlet-to-Neumann (DN) map associated with the nonlinear PDE agrees to the DN map of the linearized equation. Hence, results developed for inverse problems of linear PDEs can be applied to solve the inverse problems for many nonlinear PDEs. For the semilinear elliptic
	equation $\Delta u +a(x,u)=0$, the inverse problem of determining $a(\cdot, \cdot)$ was investigated in \cite{16,  48} for dimension $n\geq 3$, and in \cite{13,15,48} for $n=2$.  
	Moreover,  some inverse problems have been studied for quasilinear elliptic equations  in 
	\cite{2,  18,  36,  42, 46,  49}, for  degenerate elliptic $p$-Laplacian equations in \cite{1,  45},  for fractional semilinear Schr\"odinger equations in \cite{24, 25,  38},    and etc. 
	The Calder\'on type inverse problems for quasilinear PDEs on Riemannian manifolds was recently investigated in  \cite{30}  by using the Poisson embedding approach. 
	Furthermore,   we refer the readers to \cite{47,  50} for more relevant discussions on inverse problems of nonlinear elliptic equations in the existing developments. Recently, an important method was proposed to study inverse problems for semilinear elliptic equations, which is referred to as the \emph{higher order linearization}, and this method has been applied to tackle some challenging inverse problems (\cite{9, 10, 11,  20,  21, 26,  27, 37}).  
	
	The inverse problems for nonlinear hyperbolic equations have also attracted a lot of attention. It turns out that the nonlinear interaction of waves can generate new waves, which are actually beneficial in solving the related inverse problems. In \cite{23}, it was shown that the local measurements may uniquely recover global topology and differentiable structure, and the conformal class of the metric $g$ on a globally hyperbolic $4$-dimensional Lorentzian manifold,  for a wave equation with a quadratic nonlinearity. In \cite{32}, inverse problems were investigated for more general semilinear wave equations on Lorentzian manifolds, and in \cite{31},  analogous inverse problems were studied for the Einstein-Maxwell equations. We refer to \cite{3,   12, 22,  28,   29,  51} and rich references therein for more related studies of inverse problems for hyperbolic PDEs.
	
	The inverse problems mentioned  above are mainly concerned with recovering  coefficients of the underlying nonlinear PDEs through \emph{active measurements}.   In the physical scenario, the PDE coefficients correspond to the unknown medium parameters. The active measurements mean that one actively inputs a certain source into the underlying PDE system to generate the output for the corresponding inverse problem. The input-output pair constitutes a typical measurement data set for many inverse problems including wave probing, nondestructive testing and medical imaging. On the other hand, many inverse problems make use of \emph{passive measurements}, where the measurement data are generated by an unknown source. Inverse problems with passive measurements are usually referred to as the inverse source problems, since the unknown sources are the target objects to be recovered. Typical inverse source problems include those ones from the hazardous radiation detection and the cosmological searching. Recently, the inverse problems by using passive measurements to simultaneously detect the unknown sources and the surrounding mediums, have received considerable studies in the literature, due to their strong backgrounds of practical applications including photo-acoustic and thermo-acoustic tomography \cite{40}, brain imaging \cite{7}, geomagnetic anomaly detection \cite{5,  6} and quantum mechanics \cite{33, 34}. In fact, in order to achieve the desired simultaneous recovery results, the use of both passive and active measurements was proposed for some of those inverse problems \cite{33, 34}. 
	
	Motivated by the studies discussed above, we investigate in this paper inverse boundary problems associated with a time-dependent semilinear hyperbolic  equation, where both nonlinearity and  sources are unknown. The sources include  initial displacement and initial velocity of the nonlinear wave field. It is emphasized that semilinear term considered in our study is more general than those considered in the aforementioned literature on inverse problems for nonlinear hyperbolic equations. 
	In fact,  the semilinear terms in our study may contain zeroth- and first-order terms (with respect to the underlying wave function), and both of them may be unknown. It is worth mentioning that this also constitutes one of the novel points of our study compared to most of the existing studies. In the physical situation, the zeroth-order term is in fact a certain source of the hyperbolic system. However, in order to unify and ease the exposition, we mainly refer to the initial data as the sources in our study. We establish in several generic scenarios that one can uniquely determine  the nonlinearity or/and the sources by using passive or/and active boundary observations. The major findings can be briefly summarized as follows:
	
	\begin{itemize}
		\item[(1)] When the nonlinearity is known, by using the passive measurement, we can establish a quantitative uniqueness result in determining the initial displacement and  initial velocity of the wave field from the partial boundary measurement.
		
		\item[(2)] When the nonlinearity is unknown, but belonging to a certain general class, we can also establish the qualitative uniqueness by using passive measurements to determine the initial displacement and initial velocity of the wave field.

		\item[(3)] When the initial and boundary data are small enough, and the coefficients  are admissible (see Definition \ref{Def: admissible coefficients}), one can simultaneously recover the initial data as well as the nonlinearity by using the active measurement. 
	\end{itemize}
	
	It turns out that the study for simultaneous recovery of both the sources and the nonlinearity becomes radically much more challenging than the case for recovering one of them by assuming the other is known. Finally, we would like to briefly discuss the technical novelties and developments in our study. The high order linearization technique and the nonlinear wave interaction technique mentioned earlier critically rely on the small inputs for  inverse problems. The nonlinearity shall successively generate higher order terms (with respect to certain asymptotically  small parameters) that can provide more information for the inverse problems. We shall develop techniques following a similar spirit in tackling  new inverse problems.   On the other hand, it is known that one salient feature for the nonlinear hyperbolic system is the finite-time blowup of  solutions. If certain conditions are fulfilled, the  blowup may be avoided  through boundary inputs in the context of PDE controls \cite{8}. In this paper, we shall also make use of the controllability properties for semilinear wave equations in studying the associated inverse problems. We believe that the mathematical strategy developed in the current article can be extended to attack other challenging inverse problems in different contexts. Very recently, we also investigate similar problem for nonlinear parabolic systems, and we refer readers to \cite{39} for further discussions.
	
	The rest of this paper is organized as follows. In Section \ref{Sec 2}, we state the main results for the inverse problems. In Section \ref{Sec 3},   the well-posedness on the initial-boundary value problems of the semilinear hyperbolic equations within different settings of nonlinearities are studied. Section \ref{Sec 4} is devoted to the  determination of  the initial data by using control methods for the hyperbolic equations. In Section \ref{Sec 5},  an approximation property for the linear wave equations is established. Furthermore,  the higher order linearization technique  is developed to prove the uniqueness of determining both the nonlinearity and the initial data. Finally, in Appendix \ref{Section A}, we present the complex geometrical optics solutions for linear wave equations, which are needed in the proof of the main results.

	\section{Statement of  main results}\label{Sec 2}

	Let $\Omega\subseteq\mathbb{R}^n$ be a nonempty  bounded  domain
	with a smooth boundary $\Gamma$, for $n\geq 2$.    Assume that $\Gamma_0$ is a  relatively 
	open  subset  of $\Gamma$. Denote by $\nu=(\nu_1, \cdots ,\nu_n)$  the unit outer normal vector on $\Gamma$.     For  any  $T>0$,   
	set 
	$
	Q=\Omega\times(0,  T)$ and   $\Sigma=\Gamma\times(0,  T).
	$  
	Consider the following initial-boundary value problem of the semilinear wave equation: 
	\begin{eqnarray}\label{eq:wave1}
		\begin{cases}
			u_{tt}-
			\Delta  u+f(x,t,u)=0  &\text{ in } Q,\\
			u=h   &\text{ on }\Sigma,\\
			u(x, 0)=\varphi(x) , \   u_t(x, 0)=\psi(x)   &\text{ in } \Omega,
		\end{cases}
	\end{eqnarray}
	where $(\varphi, \psi)$  is a pair of initial values,  
	$h$ is a boundary  value with $\mathrm{supp}\  h\subseteq \Gamma_0\times[0, T]$
	and $f=f(x,t,s): Q\times \mathbb{R} \rightarrow \mathbb{R}$ is a given
	function,  so that  (\ref{eq:wave1})  is well-posed.  Some local  and global well-posedness results
	for  (\ref{eq:wave1}) will be given in Section \ref{Sec 3}, respectively.
	
	\smallskip

	First,  we present the first inverse  problem on determining initial values.
	For
	any  $(\varphi, \psi)\in H^1_0(\Omega)\times L^2(\Omega)$, $h=0$  and  
	a   suitable  function $f$, 
	which guarantees the global well-posedness of (\ref{eq:wave1}),    introduce the following  
	passive measurement: 
	$$
	\Lambda^0_{\varphi,\psi, f}=\partial_\nu u\Big|
	_{\Gamma_0\times (0, T)}, 
	$$ 
	where  $u$ is the solution to \eqref{eq:wave1}  associated to  $(\varphi, \psi)\in H^1_0(\Omega)\times L^2(\Omega)$
	and $h=0$, and  $\partial_\nu u$ denotes the  outer normal derivative of $u$. 
	Physically, $(\varphi, \psi, f)$ may be regarded as  unknown sources defined on $\Omega$  and 
	$Q\times\mathbb R$,  $h$ is a boundary input, and all of  them generate a wave filed 
	$(u, u_t)$. If $h=0$,  the wave field is uniquely generated by the sources $(\varphi, \psi, f)$. 
	$\Lambda^0_{\varphi,\psi, f}$ encodes the local boundary measurement   on  $\Gamma_0$ of the wave field.

	\medskip
	
	In this paper, we are first concerned with the following   inverse problem:
	
	\begin{itemize}
		\item \textbf{Inverse problem 1.} Can we identify  unknown functions $(\varphi, \psi, f)$ by using the passive measurement $\Lambda_{\varphi,\psi, f}^0$?
	\end{itemize}
	For this problem,  we  give some 
	assumptions.
	Suppose that 
	\begin{equation}\label{BBCC}
		\Gamma_0=\Big\{x\in \Gamma\    \Big| \   (x-x_0)\cdot \nu(x)> 0\Big\}\  \mbox{   for  some }
		x_0\in \mathbb R^n\setminus \overline{\Omega}.
	\end{equation}
	Assume that 
	$T>T^*$, where
	\begin{equation}\label{T conditions}
		T^*=2\max\limits_{x\in\overline{\Omega}}  |x-x_0|.	
	\end{equation}
	Also,  introduce the following increasing condition on   $f:  Q\times\mathbb R\rightarrow\mathbb R$:
	\begin{align}\label{condition of nonlinear f at infinity data}
		\limsup\limits_{s\rightarrow\infty}\displaystyle\frac{\p_s f(x, t, s)}{\mbox{ln}|s|}=0,\quad
		\mbox{ uniformly for }(x,  t)\in \overline Q,
	\end{align}
	and a set:
	\begin{eqnarray}\label{set M}
		\begin{array}{ll}
			\displaystyle\mathcal{M}_T=\Big\{ f: Q\times\mathbb R\rightarrow\mathbb R\ \Big|  &
			f(x, t, \cdot)\in C^1(\mathbb R)\mbox{ in }Q, \  f(\cdot, \cdot, 0)\in L^2(Q),\\
			&\displaystyle\mbox{ and } 
			(\ref{condition of nonlinear f at infinity data})\mbox{  holds } \Big\}.
		\end{array}
	\end{eqnarray} 
	By Section \ref{Sec 3},   for any $(\varphi, \psi)\in H^1_0(\Omega)\times L^2(\Omega)$,  
	$h=0$ and $f\in \mathcal{M}_T$, (\ref{eq:wave1}) has a unique solution 
	$$
	u\in H_0=C([0, T]; H^1_0(\Omega))\cap C^1([0, T]; L^2(\Omega)).
	$$
	Moreover,  $\partial_\nu  u\in L^2(\Sigma)$.
	Note that any function in $L^\infty(Q; W^{1, \infty}(\mathbb R))$ satisfies  (\ref{condition of nonlinear f at infinity data}).
	
	\medskip
	
	The  uniqueness  result of this paper on the  
	\textbf{Inverse problem 1}  is stated as follows.

	\begin{thm}[Stability of initial data by passive measurement]\label{Main Thm 1} 
		For any $T>T^*$,    $f\in\mathcal M_T$  and 
		$(\varphi_j, \psi_j)\in H^1_0(\Omega)\times L^2(\Omega)$ $(j=1, 2)$,   if $u_j\in H_0$  is the solution  to the
		following semilinear  wave equation: 
		\begin{align}\label{IBVP for thm 1 for j=1,2}
			\begin{cases}
				u_{j,  tt}-\Delta u_j+f(x,t,u_j)=0  &\text{ in } Q,\\
				u_j=0  &\text{ on }\Sigma,\\
				u_j(x, 0)=\varphi_j(x), \     u_{j, t}(x, 0)=\psi_j(x)  
				&\text{ in } \Omega,
			\end{cases}
		\end{align}
		then the following quantitative stability estimate holds:
		\begin{align}\label{Stability estimate in Thm 1}
			\begin{split}
				&\left\|(\varphi_1-\varphi_2, \psi_1-\psi_2)\right\|_{H^1_0(\Omega)\times L^2(\Omega)} \\
				&\leq   C(f, u_1, u_2, n, T, \Omega,  \Sigma,  \Gamma_0) \left\|\Lambda_{\varphi_1,\psi_1, f}^0-\Lambda_{\varphi_2, \psi_2, f}^0\right\|_{L^2(0,  T;  L^2(\Gamma_0))},
			\end{split}
		\end{align}
		where $ C(f, u_1, u_2, n, T, \Omega,  \Sigma,  \Gamma_0)$ denotes a positive constant depending  on $f$, $u_1$, $u_2$, $n$, $T$, $\Omega$, $\Sigma$ and $\Gamma_0$.	
	\end{thm}
	
	As a corollary of Theorem \ref{Main Thm 1},   
	introduce
	the  following set on   $f$:
	\begin{eqnarray}\label{set C}
		\mathcal{C}_{T}=\Big\{ f:  Q\times\mathbb R\rightarrow\mathbb R\  \Big|\,     
		&f(x, t, s)=f_0(x, t, s)\chi_{[0, T^*+\epsilon]}(t)+g(x,t,s)\chi_{[T^*+\epsilon, T]}(t) \\[2mm]
		&\quad\mbox{for some }\epsilon>0\mbox{  with }T^*+\epsilon<T  \mbox{ and any given }f_0\in\mathcal{M}_T,\\
		&\quad\mbox{where } g\in\mathcal{M}_T \Big\},   
	\end{eqnarray}
	where $\chi_E=\begin{cases}
		1 &\text{ if }x\in E\\
		0 & \text{ otherwise}
	\end{cases}$ denotes the characteristic function on a set $E\subseteq\mathbb R$.
	The following corollary states that when the nonlinear function $f\in \mathcal{C}_T$, then one can determine the initial data regardless of the nonlinearity $f\in \mathcal{C}_T$.
	
	\medskip

	\begin{cor}\label{Main Thm 2}
		For any $T>T^*$,   $f_j\in\mathcal C_{T}$ given by \eqref{set C},
		and $(\varphi_j, \psi_j)\in H^1_{0}(\Omega)\times L^2(\Omega)$ $(j=1, 2)$.
		Let $u_j\in H_0$  be the solution  to the
		following semilinear  wave equation: 
		\begin{align}\label{QQ}
			\begin{cases}
				u_{j, tt}-\Delta u_j+f_j(x,t,u_j)=0  &\text{ in } Q,\\
				u_j=0  &\text{ on } \Sigma,\\
				u_j(x, 0)=\varphi_j(x) , \      u_{j, t}(x, 0)=\psi_j(x) &\text{ in }\Omega,
			\end{cases}
		\end{align}
		and	\begin{equation}\label{eq:a1}
			\Lambda^0_{\varphi_1,\psi_1,f_1}=\Lambda^0_{\varphi_2, \psi_2,f_2},
		\end{equation}
		then 
		$$
		(\varphi_1, \psi_1)=(\varphi_2,\psi_2)\quad\mbox{  in  }\Omega.
		$$	
		This means that the passive  measurement 
		$\Lambda^0_{\varphi,\psi, f}$ uniquely determines $(\varphi,\psi)$, independent of  
		functions $f$ in $\mathcal C_{T}$.
	\end{cor}

	\begin{rmk} Let us remark that:
		\begin{itemize}
			\item[(1)] Theorem $\ref{Main Thm 1}$ and  Corollary $\ref{Main Thm 2}$
			show that one may use the passive  measurement 
			to determine  initial data $(u(0), u_t(0))=(\psi, \psi)$, respectively, 
			for a given   $f\in \mathcal{M}_T$ or  any $f \in \mathcal C_{T}$,
			where $\mathcal{M}_T$ and $\mathcal C_{T}$ 
			are given in \eqref{set M} and \eqref{set C}, respectively. 
			
			\item[(2)] Note that  the requirements on nonlinear functions $f$ in 
			\eqref{set M} and \eqref{set C} are technical.
			The known observability result for linear wave equations is used to 
			prove the above uniqueness results. 
			When $f$ is fixed,   the conditions on     \eqref{BBCC},  \eqref{set M}
			and $T>T^*$ assure  the coefficient $a\in L^\infty(0, T; L^n(\Omega))$ 
			for  the linearized system of \eqref{eq:wave1}:
			$$\wt u_{tt}-\Delta \wt u+a(x, t)\wt u(x, t)=0,$$
			where $a$  is given in \eqref{AAAA}. It
			satisfies the regularity   requirement  in the observability result.  
			When $f$ is unknown,  it  is chosen in the set \eqref{set C}. 
			Indeed, this condition   divides $f$ into two parts with respect to time. 
			In  the first time interval $[0,  T^*+\epsilon]$, 
			we  may identify initial  data by the passive measurement  by Theorem $\ref{Main Thm 1}$. 
			Hence,  there is indeed no requirement on nonlinear function  $g$ in the rest time interval $[T^*+\epsilon, T]$ only if it ensures 
			the well-posedness. 
		\end{itemize}
	\end{rmk}
	\begin{rmk}
		The results  on inverse problems
		in Theorem $\ref{Main Thm 1}$ and  Corollary $\ref{Main Thm 2}$  may be 
		generalized to the following semilinear  hyperbolic equation:
		\begin{eqnarray*}
			\begin{cases}
				u_{tt}-\nabla\cdot(\sigma\nabla u)
				+f(x,t,u)=0  &\text{ in } Q,\\
				u=h   &\text{ on }\Sigma,\\
				u(x, 0)=\varphi(x) , \   u_t(x, 0)=\psi(x)   &\text{ in } \Omega,
			\end{cases}
		\end{eqnarray*}
		where $\sigma(\cdot)\in  C^2(\overline{\Omega};
		\mathbb R^{n\times  n})$ is a positive definite matrix-valued function, which satisfies  the  following  condition:
		\medskip
		\begin{itemize}
			\item[(\textbf{H})] 	{\it There exists a positive constant $\rho_0$ 
				and a positive  function $d(\cdot)\in C^2(\overline{\Omega})$
				without any critical point in $\overline{\Omega}$,  such that for any 
				$\LC x, \xi^1, \cdots, \xi^n\RC \in\overline{\Omega}\times\mathbb R^n$,
				\begin{align}\label{condition on sigma}
					\sum\limits_{i, j=1}^n\sum\limits_{i',  j'=1}^{n}\Big[  2\sigma^{i  j'}(\sigma^{i' j} d_{x_{i'}})_{x_{j'}}-
					\sigma^{i j}_{x_{j'}} \sigma^{i' j'} d_{x_{i'}}
					\Big] \xi^i\xi^j\geq \rho_0 \sum\limits_{i, j=1}^n \sigma^{i j}(x)\xi^i\xi^j.
				\end{align}
			}
		\end{itemize}
		Also,  
		\begin{equation}\label{lll}
			\Gamma_0=\left\{x\in \Gamma\  \Big| \, \sum\limits_{i, j=1}^n \sigma^{i j}(x)d_{x_i}(x)\nu_j(x)> 
			0\right\}.
		\end{equation}
		The above conditions on $\sigma$ and $\Gamma_0$ are used to guarantee 
		the observability of linear hyperbolic equations.
	\end{rmk}
	
	\medskip

	More importantly,  in this paper we will determine  
	coefficient  and initial data  simultaneously for the semilinear  wave  equation
	(\ref{eq:wave1}).  To  avoid confusion of notations,  
	we replace $f$ by $\tilde f$ and consider the 
	following semilinear wave equation:
	\begin{align}\label{eq:wave simul}
		\begin{cases}
			u_{tt}- \Delta u +\tilde f(x,t,u)=0  &\text{ in } Q,\\
			u= h   &\text{ on }\Sigma,\\
			u(x,  0)=\varphi(x), \quad   u_t(x, 0)=\psi(x)    &\text{ in } \Omega.
		\end{cases}
	\end{align}  
	For any given  pair of initial values
	$(\varphi, \psi)$ and a suitable  function $\tilde f$,  which
	guarantees the well-posedness of (\ref{eq:wave simul}),  
	define the following  the input-output map by
	\begin{align}\label{input-output semilinear}
		\Lambda_{\varphi,\psi,  \tilde f}(h) =\left(
		\partial_\nu  u\Big |_{\Sigma}, \ u(\cdot, T), \  
		u_t(\cdot, T)\right), \quad \text{ for all }   h\in E_\delta, 
	\end{align} 
	where $E_\delta$ will be defined later  and $u$  is the solution to 
	(\ref{eq:wave simul}) associated to $(\varphi, \psi, h)$.
	If $\Lambda_{\varphi, \psi, \tilde f}(h)$ is known for all $h\in E_\delta$, 
	it means that the  operator $\Lambda_{\varphi,\psi, \tilde f}$ 
	is known and it is called the  active measurement.  We are concerned with the following 
	inverse problem.
	
	\begin{itemize}
		\item \textbf{Inverse problem 2.} Can we identify  unknown initial 
		data and coefficient $(\varphi,\psi,  \tilde f)$ by using the active measurement 
		$\Lambda_{\varphi,\psi, \tilde f}$?
	\end{itemize}
	
	To our best knowledge, 
	this simultaneously recovering inverse problem for  semilinear wave equations is the
	first result to be considered in the field.

	First,  introduce some notations and
	assumptions. 
	Assume that   $m$ is a  
	positive integer and  define the  energy space $E^m$ as in \cite[Section 1]{28}  
	and \cite[Definition 3.5 in page 596]{4}: 
	\begin{align}\label{function space E^m}
		E^m = \bigcap _{ k=0 }^m C^k ([0,T];H^{m-k}(\Omega)),
	\end{align}
	which is equipped with the  norm $\norm{\cdot}_{E^m}$ as
	\begin{align*}
		\norm{u}_{E^m}=\sup_{0\leq  t\leq T} \sum_{k=0}^m \Big\|\p_t ^k u(\cdot, t)\Big\|_{H^{m-k}(\Omega)},
		\quad \forall\   u\in E^m.
	\end{align*}

	Inspired by \cite[Definition 1]{28}, we impose
	the following conditions on  $\tilde f$.
	
	\begin{defi}[Admissible coefficients]\label{Def: admissible coefficients}
		For any $T>T^*$ $($in $(\ref{T conditions}))$, 
		$\tilde f=\tilde f(x,t,s)$$: 
		Q\times\mathbb R\rightarrow \mathbb R$ is called an admissible 
		coefficient, if  it satisfies:
		\begin{itemize}
			\item[(1)] Analyticity on $\mathbb R$:	
			\begin{align}\label{condition of the nonlinear term f(x,t,z)}
				\begin{cases}
					\mbox{ the map $ s\mapsto \tilde f(\cdot,\cdot,  s)$
						is analytic  on $\mathbb  R$ with values in $E^{m+1}$},  \\
					\ \tilde f(x,t,0)=0, \mbox{ in } Q.
				\end{cases}
			\end{align}
			This means that  $\tilde f$ may be written  as the Taylor expansion at 
			any $s_0\in 
			\mathbb R$:
			\begin{equation*}
				\tilde f(x,t,s)= \sum_{k=0}^\infty \tilde f^{(k)}(x,t,s_0) \frac{(s-s_0)^k}{k!},
			\end{equation*} 
			where $\frac{\tilde f^{(k)}(x,t,s_0)}{k!}=\frac{\partial_s^k
				\tilde f(x,t,s_0)}{k!}$ are Taylor's coefficients at $s_0\in \mathbb R$, 
			for any $k\in \N$.
			
			\item[(2)] Compact support: For any $s\in\mathbb R$, 
			\begin{equation}\label{BDB}
				\supp\   \tilde   f(\cdot, \cdot, s) \subseteq \overline{\Omega} \times [t_1, t_2],
			\end{equation}
			where   $t_1$ and $t_2$ are two positive  constants  with $T^*<t_1<t_2<T$.	\end{itemize}
		
	\end{defi}
	
	\begin{rmk}
		The definition of  admissible coefficients 
		is inspired by the need to guarantee the  well-posedness  of  $(\ref{eq:wave simul})$  in $E^{m+1}$ and the application of  complex 
		geometrical optics solutions $($see Section $\ref{Sec 5})$. 
		Indeed,  in order to derive the well-posedness results,   the compact support condition on 
		$\tilde f$ may be weaken to 
		$$\supp \, \tilde  f(\cdot, \cdot, s) \subseteq \overline{\Omega} \times (0,   T].$$
		It suffices to require $\tilde  f$ to be zero  near  initial time for the compatibility
		conditions.  The compact support condition   $(\ref{BDB})$ is  technical and it will 
		be used in 
		studying
		the above inverse problem.
	\end{rmk}

	Furthermore,   for a positive  integer  $m$,  define the following  function space:
	\begin{equation}\label{BB}
		\mathcal{N}_m=\Big\{ h\in H^m(\Sigma)\ \Big|\  h\in H^{m-k}_0(0, T;  H^k(\Gamma)),\ 
		\text{ for } k=0, 1, \cdots,  m-1\Big\},
	\end{equation}
	and for  a positive constant $\delta$,  set
	$$
	E_\delta=\Big\{h\in \mathcal{N}_{m+1}\  \Big|\  \norm{h}_{H^{m+1}(\Sigma)}<\delta/2 \Big\}.
	$$
	
	By the local well-posedness of the semilinear wave equation \eqref{eq:wave simul}
	(see Section \ref{Sec 3}),  for any  $m>
	n+1$  and an admissible coefficient $\tilde f$,  there exists a $\delta>0$,   when   $(\varphi, \psi)\in  H^{m+1}_0(\Omega)
	\times  H^m_0(\Omega)
	$  and  $h\in \mathcal{N}_{m+1}$ satisfy  
	$$\norm{(\varphi, \psi)}_{H^{m+1}(\Omega)\times  H^{m}(\Omega)}+ \norm{h}_{H^{m+1}(\Sigma)}<\delta,$$
	\eqref{eq:wave simul}  has a unique  solution $u\in E^{m+1}$ and   
	$\partial_\nu u\in  H^m(\Sigma)$.  
	
	\medskip

	Now, we give an answer to \textbf{Inverse problem 2} as follows.

	\begin{thm}[Simultaneous recovery by active measurement]\label{Main Thm:Simultaneous}
		Assume that $T>T^*$,   $m>n+1$, and  $\tilde f_1$  and $\tilde f_2$  are 
		two admissible coefficients. There  exists  a  $\delta>0$,  such that for  any  $(\varphi_j, \psi_j)\in H_0^{m+1}(\Omega)\times H^m_0(\Omega)$ $(j=1, 2)$ with
		$\left\|(\varphi_j, \psi_j)\right\|_{H^{m+1}(\Omega)\times  H^{m}(\Omega)}<\delta/2$,   
		denote by $u_j\in E^{m+1}$  the solution to  the following semilinear  wave equation: 
		\begin{align}\label{IBVP of simultaneous recovery}
			\begin{cases}
				u_{j,  tt}-\Delta u_j +\tilde f_j(x,t,u_j)=0  &\text{ in } Q,\\
				u_j= h   &\text{ on }\Sigma,\\
				u_j(x,  0)=\varphi_j(x), \   u_{j, t}(x, 0)=\psi_j(x)    &\text{ in } \Omega,
			\end{cases}
		\end{align}
		for $j=1,2$. Let $\Lambda_{\varphi_j , \psi_j, \tilde f_j}$ be the input-output map defined via \eqref{input-output semilinear} for $j=1,2$, and if 
		$$\Lambda_{\varphi_1 , \psi_1, \tilde f_1}(h)=\Lambda_{\varphi_2, \psi_2,  \tilde f_2}(h), \quad \text{ for all }h\in E_\delta,$$
		then
		$$
		\varphi_1 =\varphi_2,  \  \psi_1=\psi_2 \text{ in } \Omega\quad \text{ and }
		\quad \tilde f_1=\tilde f_2 \text{ in }Q\times \mathbb{R}.
		$$
	\end{thm}

	The proof of Theorem \ref{Main Thm:Simultaneous} is mainly based on the higher order linearization  method, which was initiated in \cite{23} for  some  nonlinear hyperbolic equations. Recently, this method 
	has been extended to many  other different problems, such as \cite{20, 21, 26, 27} and
	rich references therein.

	\begin{rmk}
		Now,  we explain  main differences between Theorem $\ref{Main Thm 1}$ and Theorem $\ref{Main Thm:Simultaneous}$.
		\begin{itemize}
			\item[(1)] In Theorem $\ref{Main Thm 1}$, any initial value 
			$(\varphi,\psi)$ $($without smallness conditions$)$ of the initial boundary value problem \eqref{eq:wave1} can be determined uniquely, by utilizing the passive  measurement, under suitable assumptions on  coefficients $f$.
			
			\item[(2)] In Theorem $\ref{Main Thm:Simultaneous}$, by using the active measurement, one can determine 
			small initial data $(\varphi,\psi)$ and admissible  coefficient $\tilde  f$ simultaneously 
			in the initial boundary value problem $(\ref{eq:wave simul})$.
			The smallness conditions in Theorem $\ref{Main Thm:Simultaneous}$ are mainly used to prove the local well-posedness. In the determination of initial data and coefficients,  they are not essential.

		\end{itemize}
	\end{rmk}
	
	Before ending this section,  we give  a corollary for Theorem \ref{Main Thm:Simultaneous}  to show the simultaneously recovering for the following linear wave equation: 
	\begin{align}\label{IBVP of simultaneous recovery linear intr}
		\begin{cases}
			u_{  tt}-\Delta u +q u=0  &\text{ in } Q,\\
			u= h   &\text{ on }\Sigma,\\
			u(x,  0)=\varphi(x), \   u_{ t}(x, 0)=\psi(x)    &\text{ in } \Omega,
		\end{cases}
	\end{align}
	where $q\in E^{m+1}$, $\varphi\in H_0^{m+1}(\Omega)$, $\psi\in H_0^m(\Omega)$ and $h\in \mathcal{N}_{m+1}$ for $m>n+1$.
	By   Lemma \ref{lem*},    (\ref{IBVP of simultaneous recovery linear intr})
	is  well-posed with $ u\in  E^{m+1} $ and $\partial_\nu u\in H^m(\Sigma)$.
	Due to linearity,  we do not need to impose any smallness condition for both initial data and boundary inputs.

	Now,  for any $\varphi\in H_0^{m+1}(\Omega)$, $\psi\in H_0^m(\Omega)$  and  
	$q\in E^{m+1}$,   we  define the corresponding input-output map $\Lambda_{\varphi,\psi,q}$ of  \eqref{IBVP of simultaneous recovery linear intr} via 
	\begin{align}\label{input-output linear}
		\Lambda_{\varphi,\psi,  q}(h) =\left(
		\partial_\nu  u\Big |_{\Sigma}, \ u(\cdot, T), \  
		u_t(\cdot, T)\right), \quad \text{ for all } h \in \mathcal{N}_{m+1},
	\end{align} 
	where $u\in E^{m+1}$ is the  solution to (\ref{IBVP of simultaneous recovery linear intr}).   In order to study this inverse problem by Theorem  \ref{Main Thm:Simultaneous},   we  still assume that 
	$t_1$ and $t_2$ are two positive  constants  with $T^*<t_1<t_2<T$ as in Definition \ref{Def: admissible coefficients}, and $q\in E^{m+1}$ with $\supp \, q\subseteq \Omega \times [t_1,t_2]$. Then the following result holds, which may be regarded as a corollary 
	of  Theorem  \ref{Main Thm:Simultaneous}  in the case that $\tilde f(x,t,u)=q(x,t)u$.
	
	\begin{cor}[Simultaneous recovery for linear wave equations]\label{Main Thm:Simultaneous linear}
		Assume that $T>T^*$,   $m>n+1$ and $q_j\in E^{m+1}$ with $\supp \, q_j\subseteq \Omega \times [t_1,t_2] $ for $j=1,2$.
		For  any  $(\varphi_j, \psi_j)\in H_0^{m+1}(\Omega)\times H^m_0(\Omega)$ $(j=1, 2)$, denote by $u_j\in E^{m+1}$  the solution to  the following linear wave equation: 
		\begin{align}\label{IBVP of simultaneous recovery linear}
			\begin{cases}
				u_{j,  tt}-\Delta u_j +q_j u_j=0  &\text{ in } Q,\\
				u_j= h   &\text{ on }\Sigma,\\
				u_j(x,  0)=\varphi_j(x), \   u_{j, t}(x, 0)=\psi_j(x)    &\text{ in } \Omega,
			\end{cases}
		\end{align}
		for $j=1,2$. Let $\Lambda_{\varphi_j , \psi_j, q_j}$ be the input-output map defined via \eqref{input-output linear} for $j=1,2$,and if 
		$$\Lambda_{\varphi_1 , \psi_1, q_1}(h)=\Lambda_{\varphi_2, \psi_2,  q_2}(h), \quad \text{ for all }h\in E_\delta,$$
		then
		$$
		\varphi_1 =\varphi_2,  \  \psi_1=\psi_2 \text{ in } \Omega\quad \text{ and }
		\quad q_1=q_2 \text{ in }Q.
		$$
	\end{cor}

	\section{Well-posedness of  semilinear wave equations}\label{Sec 3}
	
	This section  is  devoted to  investigating  the well-posedness of the semilinear wave equations \eqref{eq:wave1} and  \eqref{eq:wave simul}.  
	The global well-posedness for \eqref{eq:wave1} under the superlinear increasing condition (\ref{condition of nonlinear f at infinity data}) and local well-posedness for \eqref{eq:wave simul}  under  admissible 
	coefficients conditions are established, respectively. Throughout this paper, we denote by $C$ a  positive constant,  which is independent of solutions to involved equations  and may be different from line  to line.  Furthermore, if the constant $C$ depends on some factor, for example, $C$ depends on some real number $p$,  we will write $C=C(p)$.

	\subsection{Local  well-posedness with small data}
	This subsection is devoted to the well-posedness of the semilinear wave equation
	\eqref{eq:wave simul}.   Similar  results have been investigated in 
	some known works for   different structures on  $\tilde f$, 
	see for instance \cite{43}.

	To begin with, recall the definition  of  the  energy space 
	\begin{align*}
		E^m = \bigcap _{ k=0 }^m C^k ([0,T];H^{m-k}(\Omega)).
	\end{align*}
	By the Sobolev embedding and  \cite[Definition 3.5]{4},   the  above  space $E^m$ is  an algebra, due to 
	\begin{align*}
		\norm{\phi \psi}_{E^m} \leq C_m \norm{\phi}_{E^m} \norm{\psi}_{E^m},  \quad \text{ for any }  \phi, \psi \in E^m,
	\end{align*}
	for  any integer $m>n+1$.
	Indeed, the algebra property of  function spaces plays an essential role in the study of 
	the well-posedness for many nonlinear partial differential equations. For example,  in \cite{10,  20, 26,  27},  suitable  H\"older continuous spaces were utilized to prove the well-posedness for   semilinear elliptic equations. 
	
	\medskip
	
	Next, we recall  a known well-posedness result for the following wave equation:
	\begin{align}\label{linear-hyper in well-posed}
		\begin{cases}
			v_{tt}-\Delta v=g & \text{ in }Q,\\
			v=h & \text{ on }\Sigma,\\
			v(x, 0)=\varphi(x), \quad  v_t(x,  0)=\psi(x)  &\text{ in }\Omega.	
		\end{cases}
	\end{align}
	The  compatibility conditions up to order $m$\footnote{One needs to check the compatible conditions for $\p_t^kh(x,0)$ for $x\in \Gamma$ and for $k=0,1,\ldots, m$ to get higher order regularity estimates.}  mean that
	\begin{align}\label{compatibility conditions}
		\begin{cases}
			h(x, 0)=\varphi(x), \  
			h_t(x,0)=\psi(x)\  \mbox{on }\Gamma, \\[2mm]
			h_{tt}(x,0)=\Delta \varphi(x)+g(x,0)\  \mbox{ on }\Gamma,\\[2mm]
			h_{ttt}(x, 0)=\Delta\psi(x)+g_t(x, 0)\  \mbox{ on }\Gamma,\\[2mm]
			h_{tttt}(x, 0)=\Delta^2\varphi(x)+\Delta g(x, 0)+g_{tt}(x, 0)\  \mbox{ on }\Gamma,\\[2mm]
			\mbox{and    higher order derivatives of  $h$ up to 
				order $m$ with respect to time. }
		\end{cases}
	\end{align}
	By \cite[Theorem 2.45]{17}, the following well-posedness result holds for 
	(\ref{linear-hyper in well-posed}).

	\begin{lem}\label{lem}	Let $m$ be a non-negative  integer and $T>0$. 
		For any   $\vphi\in H^{m+1}(\Omega)$, $\psi \in H^m(\Omega)$, $h\in H^{m+1}(\Sigma)$  and 
		$g\in L^1(0, T; H^m(\Omega))$  with 
		$\partial_t^m g\in L^1(0, T; L^2(\Omega))$,  if the compatibility conditions 
		$(\ref{compatibility conditions})$ hold,  
		\eqref{linear-hyper in well-posed} admits  a unique solution 
		$$v\in C([0, T]; H^{m+1}(\Omega))\cap 
		C^{m+1}([0, T]; L^2(\Omega))$$ and $\partial_\nu v\in H^m(\Sigma)$. 
		Moreover,  
		\begin{align*}
			\begin{split}
				&\norm{v}_{C([0, T]; H^{m+1}(\Omega))}
				+\norm{v}_{C^{m+1}([0, T]; L^2(\Omega))}
				+\norm{\partial_\nu v}_{H^m(\Sigma)}\\[2mm]
				\leq &
				Ce^{CT}\Big(
				\norm{g}_{L^1(0, T; H^m(\Omega))}+
				\norm{\partial_t^m g}_{L^1(0, T; L^2(\Omega))}\\[2mm]
				&\quad\quad\quad\quad
				+\norm{\varphi}_{H^{m+1}(\Omega)}
				+\norm{\psi}_{H^m(\Omega)}
				+\norm{h}_{H^{m+1}(\Sigma)}
				\Big).
			\end{split}
		\end{align*}
	\end{lem}

	Based on Lemma \ref{lem},   we have the well-posedness result 
	for the following linear  wave  equation: 
	\begin{align}\label{linear-hyper in well-posed11}
		\begin{cases}
			v_{tt}-\Delta v+qv=g_1 & \text{ in }Q,\\
			v=h & \text{ on }\Sigma,\\
			v(x, 0)=\varphi(x), \  v_t(x,  0)=\psi(x)  &\text{ in }\Omega,		
		\end{cases}
	\end{align}
	where 
	$q\in E^m$. 
	\begin{lem}\label{lem*}
		Let $m>n+1$  and $T>0$. 
		For any   $\vphi\in H^{m+1}_0(\Omega)$, $\psi \in H^m_0(\Omega)$, 
		$h\in \mathcal{N}_{m+1}$ $($see $(\ref{BB}))$,  $q\in E^m$,  and 
		$g_1\in E^m$  with 
		$\partial_t^k g_1(\cdot, 0)\in H^{m-k}_0(\Omega)$ for 
		$k=0, 1, \cdots, m-2$,   
		\eqref{linear-hyper in well-posed11} admits a unique solution 
		$$v\in E^{m+1}\  \mbox{ and }\   \partial_\nu v\in H^m(\Sigma).$$ 
		Moreover, 
		\begin{align}\label{QQQQ}
			\begin{split}
				&\displaystyle\norm{v}_{E^{m+1}}
				+\norm{\partial_\nu v}_{H^m(\Sigma)}\\[2mm]
				\displaystyle\leq &
				Ce^{CT}\left(
				\sum\limits_{k=0}^{m}\norm{\partial_t^k g_1}_{L^1(0, T; H^{m-k}(\Omega))}
				+\norm{\varphi}_{H^{m+1}(\Omega)}
				+\norm{\psi}_{H^m(\Omega)}
				+\norm{h}_{H^{m+1}(\Sigma)}
				\right).
			\end{split}
		\end{align}
	\end{lem}

	\begin{proof}
		First,  we consider the case of $q=0$.
		By Lemma \ref{lem} and the definition of $E^{m+1}$, 
		it suffices to prove that for any positive integer $k\in (0, m+1)$, 
		$$
		\partial_t^k v(\cdot, \cdot)\in C([0, T]; H^{m+1-k}(\Omega)).
		$$
		Set $u=\partial_t^k v$ and  it satisfies the following  equation: 
		\begin{align*}
			\begin{cases}
				u_{tt}-\Delta u=\partial_t^k g_1& \text{ in }Q,\\
				u=\partial_t^k h & \text{ on }\Sigma.	
			\end{cases}
		\end{align*}
		Since
		$$
		g_1\in  E^m, \quad h\in H^{m+1}(\Sigma) \    \mbox{ and }\   
		(\varphi, \psi)\in H^{m+1}(\Omega)\times  H^m(\Omega),
		$$
		we have that
		\begin{eqnarray*}
			&&u(\cdot, 0)\in H^{m+1-k}(\Omega), \    
			u_t(\cdot, 0)\in H^{m-k}(\Omega),\  
			\partial_t^k h\in H^{m+1-k}(\Sigma),\\[2mm]
			&&\partial_t^k g_1\in L^1(0, T; H^{m-k}(\Omega))  \  \mbox{ and }\  
			\partial_t^m g_1\in L^1(0, T; L^2(\Omega)).
		\end{eqnarray*}
		And the compatibility conditions of order $m-k$ hold. 
		By Lemma \ref{lem}, 
		$$
		u=\partial_t^k v\in C([0, T]; H^{m+1-k}(\Omega)) \  \mbox{  for any integer  }
		k\in (0, m+1).
		$$
		Also,  the estimate (\ref{QQQQ}) remains  true for the case of $q=0$.
		
		\medskip

		Next,  consider the general case of    $q\in E^m$.  Define the set
		$$
		\mathcal{K}_1=\Big\{
		v\in E^m\  \Big|\  
		\partial_t^k v(\cdot, 0)\in H^{m-k}_0(\Omega), \mbox{ for }k=
		0, 1, \cdots, m-2  \Big\}.
		$$  For any $v\in \mathcal{K}_1$,  consider the following wave equation:
		\begin{align}\label{EL1}
			\begin{cases}
				w_{tt}-\Delta w=g_1-qv& \text{ in }Q,\\
				w=h & \text{ on }\Sigma,\\
				w(x, 0)=\varphi(x),\  w_t(x,  0)=\psi(x)  &\text{ in }\Omega.	
			\end{cases}
		\end{align}
		By $g_1, v, q\in  E^m$,  it follows that
		$g_1-vq\in E^m$.
		Also,  
		$$
		\partial_t^k [g_1(\cdot, 0)-(qv)(\cdot, 0)]\in H^{m-k}_0(\Omega), \mbox{ for }k=
		0, 1, \cdots, m-2.
		$$
		By the  well-posedness result in  the case of $q=0$,
		\eqref{EL1} admits a unique solution 
		$$w\in E^{m+1}\  \mbox{ and }\   \partial_\nu w\in H^m(\Sigma).$$ 
		Moreover, 
		\begin{align}\label{QQ4}
			\begin{split}
				\displaystyle\norm{w}_{E^{m+1}}
				+\norm{\partial_\nu w}_{H^m(\Sigma)}\displaystyle\leq
				&Ce^{CT}\Big(
				\sum\limits_{k=0}^{m}\norm{\partial_t^k(g_1-qv)}_{L^1(0, T; H^{m-k}
					(\Omega))}\\[2mm]
				&\quad\quad\quad\quad+\norm{\varphi}_{H^{m+1}(\Omega)}
				+\norm{\psi}_{H^m(\Omega)}
				+\norm{h}_{H^{m+1}(\Sigma)}
				\Big).
			\end{split}
		\end{align}
		
		Define the mapping 
		$$\mathcal{L}_1: \mathcal{K}_1\rightarrow \mathcal{K}_1,\quad
		\mathcal{L}_1(v)=w,$$
		where $w$ is the solution to (\ref{EL1}) associated to $v\in \mathcal{K}_1$.
		For any $v_1, v_2\in \mathcal{K}_1$,  denote by $w_1$ and $w_2$ 
		the associated solutions to (\ref{EL1}).  By (\ref{QQ4}),  we obtain that
		\begin{align*}
			&\displaystyle\norm{w_1-w_2}_{E^{m+1}}
			+\norm{\partial_\nu(w_1-w_2)}_{H^m(\Sigma)}\\
			\leq
			&Ce^{CT}\sum\limits_{k=0}^{m}\norm{\partial_t^k[q(v_1-v_2)]}_{L^1(0, T; H^{m-k}
				(\Omega))}\\
			\leq &CTe^{CT}
			\sum\limits_{k=0}^{m}\norm{\partial_t^k[q(v_1-v_2)]}_{C([0, T]; H^{m-k}
				(\Omega))}\\
			=&
			CTe^{CT}
			\norm{q(v_1-v_2)}_{E^m}
			\leq  
			CTe^{CT}
			\norm{q}_{E^m}\norm{v_1-v_2}_{E^m}.
		\end{align*}
		If $T$ is sufficiently small such that $CTe^{CT}\norm{q}_{E^m}<1$, 
		then by the Banach fixed point theorem,  $\mathcal{L}_1$ has a unique 
		fixed point $v\in \mathcal{K}_1$. 
		Since \eqref{linear-hyper in well-posed11} is a linear equation, by a rescaling with respect to the time variable, we can get the result for any $T>0$.
	\end{proof}

	The main  result of this subsection is stated as follows.

	\begin{thm}[Local well-posedness]\label{thm:well-posedness}
		Assume that  $m>n+1$ and $\tilde  f$ is an  
		admissible  coefficient.  Then there exists a $\delta>0$,  such that for any $(h,\varphi,\psi)$ in the set 
		$$U_\delta= \Big\{ (h,\varphi,\psi) \in \mathcal{N}_{m+1}\times H^{m+1}_0(\Omega)\times H^m_0(\Omega)\  \Big|\   
		\|h\|_{H^{m+1}(\Sigma)}+\|\varphi\|_{H^{m+1}(\Omega)}+\|\psi\|_{H^m(\Omega)}<\delta \Big\},
		$$
		\eqref{eq:wave simul} admits a unique solution $u\in E^{m+1}$  satisfying that $\p _\nu u \in H^m(\Sigma)$ and  
		\begin{equation}\label{EEE}
			\norm{u}_{E^{m+1}}+\norm{u}_{C(\overline{Q})}+\norm{\p_\nu u}_{H^m(\Sigma)}\le C \Big(\norm{h}_{H^{m+1}(\Sigma)}+\norm{\varphi}_{H^{m+1}(\Omega)}+\norm{\psi}_{H^m(\Omega)}\Big),
		\end{equation}
		where  $C$ is  a positive constant independent of  $u, h, \varphi$ and $\psi$. Moreover, the following solution map is $C^\infty$ Fr\'echet differentiable:
		\begin{align*}
			&S: U_\delta \to E^{m+1},  \qquad (h,\varphi,\psi)\mapsto u.	\end{align*}
	\end{thm}

	\begin{proof}
		Similar to arguments  in \cite{20,  27}, we prove the well-posedness of  \eqref{eq:wave simul} by  the implicit function theorem in Banach spaces.

		First,   set 	\begin{align*}
			&X_1=\mathcal{N}_{m+1}\times H^{m+1}_0(\Omega)\times H^m_0(\Omega), \\
			&X_2=\Big\{ u\in  E^{m+1} \  \Big|\     u\big|_{\Sigma}\in \mathcal{N}_{m+1},    
			\p_\nu u\in H^{m}(\Sigma),  u(\cdot, 0)\in H^{m+1}_0(\Omega),  u_t(\cdot, 0)\in H^m_0(\Omega),  \\
			&\quad\quad\quad\quad\quad\quad\quad\quad\quad u_{tt}-\Delta  u\in E^m,
			\partial_t^k(u_{tt}-\Delta u)(\cdot, 0)\in H^{m-k}_0(\Omega), k=0, 1, \cdots, m-2  \Big\},
		\end{align*}
		where 
		$$
		\|u\|_{X_2}=
		\|u\|_{E^{m+1}}+\|\partial_\nu u\|_{H^m(\Sigma)}+\|u\|_{H^{m+1}(\Sigma)}
		+\|u_{tt}-\Delta  u\|_{E^m}.
		$$
		Note that $X_2$ is nonempty and indeed $C^\infty_0(Q)\subseteq X_2$.
		Meanwhile, let us write
		\begin{align*}
			&X_3= \Big\{  g\in E^m\   \Big|\   \partial_t^kg(\cdot, 0)\in H^{m-k}_0(\Omega),\   
			\forall \   k=0, 1, \cdots, m-2 \Big\}\times X_1. 
		\end{align*}
		Consider the following map: 
		\begin{align}\label{map of well-posedness}
			\begin{split}
				&F: X_1 \times X_2 \to X_3 ,\\
				&F(h,\varphi,\psi,u)=\LC  u_{tt} - \Delta u + \tilde f(x, t, u),\  u\big|_{\Sigma}-h, \  u(\cdot,  0)-\varphi, \    u_t(\cdot, 0)-\psi  \RC,
			\end{split}
		\end{align}
		where $(h,\varphi,\psi)\in X_1$ and $u\in X_2$.    By the condition \eqref{condition of the nonlinear term f(x,t,z)} for  $\tilde f$, 
		for any positive integer $k$  and   positive constant $R$,
		\begin{align}
			\big\|\tilde f^{(k)}(\cdot, \cdot, 0)\big\|_{E^{m+1}}\leq \frac{k!}{R^k} 
			\sup_{|s|=R} \big\|\tilde f(\cdot,\cdot, s)\big\|_{E^{m+1}}.
		\end{align}
		Since $E^{m+1}$ is an algebra, it follows that  for any $u\in X_2$, 
		\begin{align*}
			\big\|\tilde f(\cdot, \cdot, u(\cdot, \cdot))\big\|_{E^{m+1}}\leq 
			&\sum\limits_{k=0}^{\infty}\frac{1}{k!} \big\|\tilde f^{(k)}(\cdot, \cdot, 0)
			\big\|_{E^{m+1}}\norm{u}^k_{E^{m+1}}\\
			\leq& \sum\limits_{k=0}^{\infty} \frac{1}{R^k}\norm{u}^k_{E^{m+1}}\sup_{|s|=R}
			\big\|\tilde f(\cdot,\cdot, s)\big\|_{E^{m+1}}.
		\end{align*}
		Choose $R=2\LC \norm{u}_{E^{m+1}}+1\RC $. Then,  $\tilde f(\cdot, \cdot, u(\cdot, \cdot))\in E^{m+1}$ and 
		\[
		\big\|\tilde f(\cdot, \cdot, u(\cdot, \cdot))\big\|_{E^{m+1}} \leq C\sup_{|s|=2\LC \norm{u}_{E^{m+1}}+1\RC }\big\|\tilde f(\cdot, \cdot, s)\big\|_{E^{m+1}}.
		\]
		By using the definition of $X_2$,  $u_{tt}-\Delta u\in E^m$ and  therefore, 
		$u_{tt}-\Delta u+\tilde f(\cdot, \cdot, u(\cdot, \cdot))\in E^m$.
		Also, by   the admissible coefficient  condition on $\tilde f$, 
		$$
		\left.\partial_t^k\tilde f(\cdot, \cdot, u(\cdot, \cdot))\right|_{t=0}=0, \   \mbox{ for any }\ k=0, 1, \cdots, m-2,\mbox{ and }u\in X_2.
		$$
		Hence, the  map $F$ is well-defined. 	
		
		\medskip
		
		Now,  we prove that $F$ is locally bounded.
		Indeed,  for any $M>0$,   when $(h, \varphi, \psi, u)\in X_1\times X_2$ 
		with $\norm{(h, \varphi, \psi, u)}_{X_1\times X_2}\leq M$, 
		\begin{align*}
			&\norm{F(h, \varphi, \psi, u)}_{X_3}\\
			\leq &\norm{u_{tt}-\Delta u+\tilde f(\cdot, \cdot, u(\cdot, \cdot))}_{E^m}
			+\norm{u-h}_{H^{m+1}(\Sigma)}\\
			& \quad +\norm{u(\cdot, 0)-\varphi}_{H^{m+1}(\Omega)}
			+\norm{u_t(\cdot, 0)-\psi}_{H^m(\Omega)}\\
			\leq &\norm{\tilde f(\cdot, \cdot, u(\cdot, \cdot))}_{E^m}+C\norm{u}_{X_2}
			+\norm{h}_{H^{m+1}(\Sigma)}+
			\norm{\varphi}_{H^{m+1}(\Omega)}+\norm{\psi}_{H^m(\Omega)}\\
			\leq &C\sup\limits_{|s|=2(1+\norm{u}_{E^{m+1}})}
			\norm{\tilde f(\cdot, \cdot, s)}_{E^{m+1}}+C\norm{u}_{X_2}
			+\norm{h}_{H^{m+1}(\Sigma)}+
			\norm{\varphi}_{H^{m+1}(\Omega)}+\norm{\psi}_{H^m(\Omega)}\\
			\leq& C\sup\limits_{|s|=2(1+M)}
			\norm{\tilde f(\cdot, \cdot, s)}_{E^{m+1}}+(C+3)M
			<\infty.
		\end{align*}
		Next,  we verify the weak holomorphy of $F$ (see \cite[page 133]{44}).  It is 
		sufficient  that for any $(h_0,\varphi_0,\psi_0 , u_0), (h,\varphi, \psi, u)\in X_1\times X_2$,  the map 
		\[
		\lambda \mapsto F\LC (h_0,\varphi_0,\psi_0 , u_0)+\lambda (h,\varphi, \psi, u)\RC
		\]
		is holomorphic in a neighborhood of the origin in  with values in $X_3$. It  suffices to check that the map 
		\[
		\lambda \mapsto \tilde f(x,t,u_0(x,t)+\lambda u(x,t)) 
		\]
		is holomorphic in  a neighborhood of the origin in $\mathbb C$ with values in $E^{m}$. 
		This follows  from the convergence of 
		the series 
		\[
		\sum_{k=0}^\infty \frac{\tilde f^{(k)}(x,t, 0)}{k!} \Big[u_0(x,  t)+\lambda u(x,  t)\Big]^k
		\]
		in $E^{m+1}$, locally uniformly in $\lambda \in \mathbb C$. Hence, $F$  is holomorphic in $X_1\times X_2$.
		
		\medskip 
		
		Moreover,  notice that $F(0,0,0,0)=0$ and $F_u(0,0,0,0):  X_2 \to X_3$  is  defined by
		\[
		F_u(0,0,0,0)w= \left(w_{tt}-\Delta w +\tilde f_u(\cdot,\cdot,0)w, w\big|_{\Sigma}, w(\cdot, 0),  w_{t}(\cdot, 0)\right),  \quad \text{ for all }   w\in X_2.
		\]
		Indeed,  by  the definition of $X_2$,  for any $w\in X_2$,  $w_{tt}-\Delta  w\in E^{m}$. Since 	$\tilde f_u(\cdot, \cdot, 0)\in E^{m+1}$, it holds that 
		$w_{tt}-\Delta w +\tilde f_u(\cdot,\cdot,0)w\in E^m$.
		Also, by 
		the admissible coefficient condition  on $\tilde f$, 
		$$
		\partial_t^k\left[w_{tt}-\Delta w +\tilde f_u(\cdot,\cdot,0)w \right](\cdot, 0)\in H^{m-k}_0(\Omega), \quad
		\forall \   k=0, 1, \cdots, m-2.
		$$
		Hence, the map $F_u(0,0,0,0)$ is well-defined. 
		Furthermore, by the well-posedness of the linear wave equation 
		\eqref{linear-hyper in well-posed11} in Lemma \ref{lem*} with $$q=\tilde f_u(\cdot,\cdot,0)\  
		\mbox{ and }\  g_1\in \mathcal{H}=\Big\{  g\in E^m\   \Big|\   \partial_t^kg(\cdot, 0)\in
		H^{m-k}_0(\Omega),\   \forall \   k=0, 1, \cdots, m-2 \Big\},$$
		$F_u(0,0,0,0)$ is a linear isomorphism from $X_2\to X_3$.   In fact, 
		for any $g_1\in\mathcal{H}$,    $h\in \mathcal{N}_{m+1}$ 
		and $(\varphi, \psi)\in H^{m+1}_0(\Omega)\times H^m(\Omega)$,
		the equation 
		\eqref{linear-hyper in well-posed11} has a unique solution  
		$w\in E^{m+1}$ and $\partial_\nu w\in H^m(\Sigma)$. 
		Also,   by the fact that $g_1\in\mathcal{H}$,  $g_1=w_{tt}-\Delta w+\tilde f_u(\cdot,\cdot,0)w$,
		$$\tilde f_u(\cdot,\cdot,0)w \in E^m\  \mbox{ and }\ 
		\partial_t^k\left[\tilde f_u(\cdot,\cdot,0)w\right](\cdot, 0)\in H^{m-k}_0(\Omega), k=0, 1, \cdots, m-2,$$
		we have that $$w_{tt}-\Delta  w\in E^m\   \mbox{  and }\    
		\partial_t^k(w_{tt}-\Delta w)(\cdot, 0)\in H^{m-k}_0(\Omega), k=0, 1, \cdots, m-2.$$

		By  the implicit function theorem in Banach spaces, 
		there exists  a $\delta>0$ and a $C^\infty$ map $S: U_\delta \to E^{m+1}$,  such that 
		for any $(h,\varphi,\psi) \in \mathcal{N}_{m+1}\times H^{m+1}_0(\Omega)\times H^m_0(\Omega)$
		satisfying  
		$$
		\norm{h}_{H^{m+1}(\Sigma)}+\norm{\varphi}_{H^{m+1}(\Omega)}+\norm{\psi}_{H^m(\Omega)}<\delta, 
		$$
		it holds 
		that 
		\[
		F(h,\varphi,\psi, S(h,\varphi,\psi))=(0,0,0,0).
		\]
		Since  $S$ is locally Lipschitz continuous and $S(0,0,0)=0$,    $u=S(h,\varphi,\psi)$ satisfies that	\[
		\norm{u}_{E^{m+1}}+ \norm{\p_\nu u}_{H^m(\Sigma)}\leq C \LC  \norm{h}_{H^{m+1}(\Sigma)}+\norm{\varphi}_{H^{m+1}(\Omega)}+\norm{\psi}_{H^m(\Omega)} \RC.
		\]
		This, combined with the Sobolev embedding theorem,  proves the local well-posedness of \eqref{eq:wave simul} and the estimate \eqref{EEE}.   
	\end{proof}

	Note that one may extend the result in  Theorem \ref{thm:well-posedness}  to more general hyperbolic equations:
	$$u_{tt}-\nabla \cdot (\sigma \nabla u)+\tilde f(x, t,  u)=0,$$
	where   $\sigma$ is either  isotropic or anisotropic. However, in the application of  a  density result of  products of solutions of linear hyperbolic equations, we simply consider the classical wave equation to demonstrate ideas of  this approach (see Section \ref{Sec 5}).

	\subsection{Global well-posedness of weak solutions}
	This subsection is devoted to the well-posedness of weak solutions to 
	the  following semilinear wave equation: 
	\begin{eqnarray}\label{global1}
		\begin{cases}
			u_{tt}-\Delta u+f(x,t,u)=0  &\text{ in } Q,\\
			u=0   &\text{ on }\Sigma,\\
			u(x, 0)=\varphi(x), \   u_t(x, 0)=\psi(x)   &\text{ in } \Omega,
		\end{cases}
	\end{eqnarray}
	where    $f$ satisfies  
	(\ref{condition of nonlinear f at infinity data}),  and the following conditions:
	$$
	f(x, t, \cdot)\in C^1(\mathbb R),\ \mbox{a.e.}\  (x, t)\in  Q \quad \mbox{and}\quad f(\cdot, \cdot, 0)\in L^2(Q).
	$$
	Under the above assumptions on $f$,  we have the following global well-posedness result for (\ref{global1}).
	
	\begin{thm}[Global well-posedness]\label{thm global}
		For any $T>0$ and $(\varphi, \psi)\in H^1_0(\Omega)\times L^2(\Omega)$,
		the semilinear wave equation $(\ref{global1})$ admits a unique solution $u$ in the class of 
		$$u\in H_0=C([0, T]; H^1_0(\Omega))\cap C^1([0, T]; L^2(\Omega))
		\   \mbox{  and  } \    \partial_\nu u\in L^2(\Sigma).$$
	\end{thm}
	
	\begin{proof} The proof is  based on  the method of the  fixed point theorem. First,  
		assume that $n\geq 3$. Set
		
		$$
		g(x, t, s)=
		\left\{
		\begin{array}{ll}
			\displaystyle\frac{f(x, t,  s)-f(x, t, 0)}{s} \quad &\mbox{ for }s\neq 0,\\[2mm]
			\partial_s f(x, t, 0) \quad   &\mbox{ for }s=0,
		\end{array}
		\right. \quad \forall\ (x, t, s)\in Q\times\mathbb R.
		$$
		For any $z\in L^\infty(0, T; L^2(\Omega))$, consider the 
		following  linear wave equation:  
		\begin{eqnarray}\label{global2}
			\begin{cases}
				u_{tt}-\Delta u+a_z(x, t)u+f(x, t,  0)=0  &\text{ in } Q,\\
				u=0   &\text{ on }\Sigma,\\
				u(x, 0)=\varphi(x), \   u_t(x, 0)=\psi(x)   &\text{ in } \Omega,
			\end{cases}
		\end{eqnarray}
		where $a_z(x, t)=g(x,t, z(x, t))$.  Then we have  that $a_z\in L^\infty(0, T; 
		L^p(\Omega))$, for  any  $p\geq 1$.
		Indeed,  by the condition (\ref{condition of nonlinear f at infinity data}), for any $\epsilon\in(0, 1)$,
		there is a $C_\epsilon>0$, such that
		$$
		|g(x, t, s)|\leq \epsilon \ln(1+|s|)+C_\epsilon,\quad\forall\ (x, t, s)\in Q\times\mathbb R.
		$$
		Therefore, for any $z\in  L^\infty(0, T; L^2(\Omega))$,
		\begin{align}\label{global3}
			\begin{split}
				&\displaystyle\sup\limits_{t\in(0, T)} \int_\Omega e^{C|a_z(x, t)|}\, dx\\
				= &
				\sup\limits_{t\in(0, T)} \int_\Omega e^{C|g(x, t, z(x, t))|}\, dx 
				\leq  \sup\limits_{t\in(0, T)} \int_\Omega
				e^{C[\epsilon\ln(1+|z(x, t)|)+C_\epsilon]}\, dx \\
				= & C(\epsilon) \sup\limits_{t\in(0, T)}
				\int_\Omega [1+|z(x, t)|]^{C\epsilon} \, dx 
				\leq  C(\epsilon)
				\sup\limits_{t\in(0, T)} \int_\Omega [1+|z(x, t)|]^2\, dx \\
				\leq &
				C(\epsilon)\left(1+\norm{z}^2_{L^\infty(0, T; L^2(\Omega))}\right)<\infty,
			\end{split}
		\end{align}
		where $\epsilon$ is a sufficiently small positive constant such that $C\epsilon\leq 2.$
		Furthermore,  similar to arguments in \cite{35},   we  may obtain 
		\begin{equation}\label{global4}
			e^{C\norm{a_z}_{L^\infty(0, T; L^{p}(\Omega))}}
			\leq C\left(
			1+\sup\limits_{t\in (0, T)}\displaystyle\int_\Omega e^{C |a_z(x, t)|}dx\right).
		\end{equation}
		Indeed,

		\begin{align*}
			&e^{C\norm{a_z}_{L^\infty(0, T; L^{p}(\Omega))}} \\
			=&\sum\limits_{j=0}^{\infty}\frac{C^j}{j!}\norm{a_z}^{j}
			_{L^\infty(0, T; L^{p}(\Omega))}\\
			\leq &
			\sum\limits_{j=0}^{p+1} 
			\frac{C^j}{j!}\sup\limits_{t\in (0, T)}
			\Big(\int_\Omega |a_z(x, t)|^pdx\Big)^{j/p}
			+\sum\limits_{j=p+1}^{\infty} 
			\frac{C^j}{j!}\sup\limits_{t\in (0, T)}
			\Big(\int_\Omega |a_z(x, t)|^pdx\Big)^{j/p}\\
			\leq &C(p)\Big[1+\sup\limits_{t\in (0, T)}\Big(\int_\Omega 
			|a_z(x, t)|^p dx\Big)^\frac{p+1}{p}\Big]+\sum\limits_{j=p+1}^{\infty} 
			\frac{C^j}{j!}\sup\limits_{t\in (0, T)}
			\Big(\int_\Omega |a_z(x, t)|^pdx\Big)^{j/p}\\
			\leq &C(p)\Big[
			1+2\sum\limits_{j=p+1}^{\infty} 
			\frac{C^j}{j!}\sup\limits_{t\in (0, T)}
			\Big(\int_\Omega |a_z(x, t)|^pdx\Big)^{j/p}
			\Big]\\
			\leq &C(p)\Big[1+2
			\sum\limits_{j=p+1}^{\infty} 
			\frac{C^j}{j!}\sup\limits_{t\in (0, T)}
			\Big(\int_\Omega |a_z(x, t)|^{j}dx\Big)\cdot |\Omega|^{\frac{j}{p}-1}
			\Big]\\
			\leq &C(p)\Big[
			1+\sum\limits_{j=0}^{\infty} 
			\frac{C_1^j}{j!}\sup\limits_{t\in (0, T)}\int_\Omega |a_z(x, t)|^{j}dx
			\Big]
			\leq  C(p)\Big(1+\sup\limits_{t\in (0, T)}
			\int_\Omega e^{C_1 |a_z(x, t)|}dx\Big),
		\end{align*}
		where $C_1=C|\Omega|^{1/p}$, $C(p)$ denotes a  positive constant, which may be different in 
		different places, and $|\Omega|$ denotes the measure of the set $\Omega$. 
		Combining (\ref{global4}) with (\ref{global3}), one has that $a_z\in 
		L^\infty(0, T; L^p(\Omega))$.
		
		\medskip
		
		Next, we prove that the linear wave equation  (\ref{global2}) 
		admits a unique solution $u\in H_0$ and 
		$\partial_\nu u\in L^2(\Sigma)$.
		For any $w\in L^1(0, T; L^{2+\gamma_0}(\Omega))$, with $\gamma_0$
		being a positive constant, which will be determined later, 
		choose $p=2(2+\gamma_0)/\gamma_0$. 
		By $a_z\in  L^\infty(0,  T; L^p(\Omega))$, 
		it holds that  $a_z w\in L^1(0, T; L^2(\Omega))$. 
		Consider the following  wave equation: 
		\begin{eqnarray}\label{EL5}
			\begin{cases}
				u_{tt}-\Delta u=-a_z(x, t)w-f(x, t,  0)  &\text{ in } Q,\\
				u=0  &\text{ on }\Sigma,\\
				u(x, 0)=\varphi(x), \   u_t(x, 0)=\psi(x)   &\text{ in } \Omega.
			\end{cases}
		\end{eqnarray}
		By Lemma  \ref{lem} with $m=0$,  
		(\ref{EL5}) has  a unique solution $u\in  H_0$ and $\partial_\nu u\in 
		L^2(\Sigma)$. By the Sobolev embedding  theorem, 
		$$
		u\in C([0, T]; L^\frac{2n}{n-2}(\Omega)).
		$$
		When $n=1$ and $n=2$, $L^\frac{2n}{n-2}(\Omega)$ is  
		replaced  by $L^\infty(\Omega)$ and $L^p(\Omega)$ for any $p\geq 1$,  
		respectively.   Choose $\gamma_0=4/(n-2)$. Then, 
		$2+\gamma_0=2n/(n-2)$ and the following mapping is well-defined:  
		$$\mathcal{L}_2: 
		L^\infty(0, T; L^{\frac{2n}{n-2}}(\Omega))\rightarrow 
		L^\infty(0, T; L^{\frac{2n}{n-2}}(\Omega)), 
		\quad \mathcal{L}_2(w)=u,$$
		where $u\in H_0$ is the solution to (\ref{EL5}). 
		Also, by \eqref{QQ},
		\begin{align*}
			\begin{split}
				&\displaystyle\norm{u}_{C([0, T]; H^{1}(\Omega))}
				+\norm{u}_{C^{1}([0, T]; L^2(\Omega))}
				+\norm{\partial_\nu u}_{L^2(\Sigma)}\\[2mm]
				\displaystyle\leq
				&Ce^{CT}\Big(
				\norm{a_zw+f(\cdot, \cdot, 0)}_{L^1(0, T; L^2(\Omega))}+
				\norm{\varphi}_{H^{1}(\Omega)}
				+\norm{\psi}_{L^2(\Omega)}
				+\norm{h}_{H^{1}(\Sigma)}
				\Big).
			\end{split}
		\end{align*}
		This implies that
		\begin{align*}
			\norm{u}_{L^\infty(0, T;  L^{\frac{2n}{n-2}}(\Omega))}
			\leq &Ce^{CT}\Big(T
			\norm{a_z}_{L^\infty(0, T; L^p(\Omega))}\norm{w}_{L^\infty(0, T; L^{\frac{2n}{n-2}}(\Omega))}+
			\norm{f(\cdot, \cdot,  0)}_{L^1(0, T; L^2(\Omega))}\\
			&\quad\quad\quad\quad
			+\norm{\varphi}_{H^{1}(\Omega)}
			+\norm{\psi}_{L^2(\Omega)}
			+\norm{h}_{H^{1}(\Sigma)}
			\Big).
		\end{align*}
		Similar to arguments in Lemma \ref{lem*},  by the Banach fixed point theorem, 
		(\ref{global2})  admits a unique solution $u\in H_0$. 
		
		\medskip
		
		Finally,   define a map 
		\begin{align*}
			\mathcal{L}_3:  L^\infty(0, T; L^2(\Omega))\rightarrow L^\infty(0, T; L^2(\Omega)),
			\quad  
			\mathcal{L}_3(z)=u, 
		\end{align*}
		for any $z\in  L^\infty(0, T; L^2(\Omega))$, where $u\in H_0$ is the solution to (\ref{global2}) associated to $a_z(x, t)=g(x, t, z(x, t))$.
		In the following, we will prove that the map  $\mathcal{L}_3$ has a unique fixed  point  in a 
		set $V\subseteq L^\infty(0, T; L^2(\Omega))$.  To this  aim,  for any $t\in[0, T]$,  set
		$$
		E(t)=\displaystyle\frac{1}{2}\int_\Omega \left[u_t^2(x, t)+
		|\nabla u|^2+u^2(x, t)\right]dx.
		$$
		Multiplying both  sides  of  the first equation  in (\ref{global2})  by $u_t$  and 
		integrating   in  $\Omega$,  one has
		\begin{align}\label{PP}
			\begin{split}
				E_t(t)=&\displaystyle-\int_\Omega
				[a_zuu_t+f(x, t, 0)u_t]dx+\int_\Omega uu_t  \, dx\\
				\leq &\displaystyle  \norm{a_z(\cdot, t)}_{L^n(\Omega)}\norm{u(\cdot, t)}_{L^{\frac{2n}{n-2}}(\Omega)}
				\norm{u_t(\cdot,t)}_{L^2(\Omega)}\\[4mm]
				&\displaystyle\quad+\norm{f(\cdot,  t, 0)}_{L^2(\Omega)}\norm{u_t(\cdot, t)}_{L^2(\Omega)}+\norm{u(\cdot, t)}_{L^2(\Omega)}\norm{u_t(\cdot, t)}_{L^2(\Omega)}.
			\end{split}
		\end{align}
		This implies that
		$$
		E_t(t)\leq C\norm{a_z(\cdot, t)}_{L^n(\Omega)}E(t)+\norm{f(\cdot,  t, 0)}_{L^2(\Omega)}
		E^{\frac{1}{2}}(t)+E(t).
		$$
		Hence, 
		\begin{align*}
			\norm{u}^2_{H_0}\leq & CE(t) \\
			\leq & C\Big[E(0)+ \norm{f(\cdot,  \cdot,  0)}^2_{L^2(Q)}\Big]
			e^{C ( 1+\norm{a_z}_{L^1(0, T; L^n(\Omega))})}\\[2mm]
			\leq & C\Big[\norm{\varphi}^2_{H^1(\Omega)}+\norm{\psi}^2_{L^2(\Omega)}
			+\norm{f(\cdot,  \cdot,  0)}^2_{L^2(Q)}\Big]
			e^{C(1+T\norm{a_z}_{L^\infty(0, T; L^n(\Omega))})}.
		\end{align*}
		By (\ref{global3}) and  (\ref{global4}), 
		it follows that for a sufficiently small $\epsilon$ with $C\epsilon\leq 1$,
		\begin{align*}
			\norm{u}^2_{H_0} 
			&\leq C\Big[\norm{\varphi}^2_{H^1(\Omega)}+\norm{\psi}^2
			_{L^2(\Omega)}
			+\norm{f(\cdot,  \cdot,  0)}^2_{L^2(Q)}\Big]
			\Big(1+\sup\limits_{t\in (0, T)}\int_\Omega e^{C |a_z(x, t)|}dx\Big) \\
			&\leq  C\Big[\norm{\varphi}^2_{H^1(\Omega)}+\norm{\psi}^2_{L^2(\Omega)}
			+   \norm{f(\cdot,  \cdot,  0)}^2_{L^2(Q)}\Big]
			\Big(1+\sup\limits_{t\in (0, T)}\int_\Omega |z(x, t)|^{C \epsilon} dx\Big) \\
			&\leq  C\Big[\norm{\varphi}^2_{H^1(\Omega)}+\norm{\psi}^2_{L^2(\Omega)}+
			\norm{f(\cdot,  \cdot,  0)}^2_{L^2(Q)}\Big]
			\Big(1+\sup\limits_{t\in (0, T)}\int_\Omega |z(x, t)| dx\Big) \\
			&\leq  C\Big[\norm{\varphi}^2_{H^1(\Omega)}+\norm{\psi}^2_{L^2(\Omega)}+ 
			\norm{f(\cdot,  \cdot,  0)}^2_{L^2(Q)}\Big]
			\Big(1+\sup\limits_{t\in (0, T)}\norm{z(\cdot, t)}_{L^2(\Omega)}\Big).
		\end{align*}
		The above estimate implies that
		\begin{align*}
			\norm{u}^2_{L^\infty(0, T;  L^2(\Omega))}
			\leq C\left[\norm{\varphi}^2_{H^1(\Omega)}+\norm{\psi}^2_{L^2(\Omega)}
			+  \norm{f(\cdot,  \cdot,  0)}^2_{L^2(Q)}\right]
			\LC 1+\norm{z}_{L^\infty(0, T;  L^2(\Omega))}\RC .
		\end{align*}
		Therefore,  there exists a positive constant $C_*$, depending  on  
		$\norm{f(\cdot,  \cdot,  0)}_{L^2(Q)}$, $\norm{\varphi}_{H^1(\Omega)}$ and
		$\norm{\psi}_{L^2(\Omega)}$, such that
		$$
		\mathcal{L}_3(V)\subseteq  V,\quad\mbox{with }\  V=\Big\{ u\in L^\infty(0, T;  L^2(\Omega))\  \Big|\   \norm{u}_{L^\infty(0, T;  L^2(\Omega))}\leq C_*  \Big\}.
		$$
		Also, $\mathcal{L}_3$  is  compact. By the Schauder  fixed point technique,  
		$\mathcal{L}_3$ has a  
		fixed point $u$ in  $V$, which is  the solution to the semilinear
		wave equation (\ref{global1}).  Since $u$ is a fixed point of $\mathcal{L}_3$,
		it is a solution to 
		(\ref{global2}) associated to some $z\in L^\infty(0, T; L^2(\Omega))$. 
		Hence, $u\in H_0$ and $\partial_\nu u\in  L^2(\Sigma)$. 
		
		Moreover, suppose that $u_1,   u_2\in C([0, T]; H^1(\Omega))\cap C^1([0, T];  L^2(\Omega))$ are two solutions to \eqref{global1}.  Set $u=u_1-u_2$. Then $u$ satisfies the 
		following  wave equation:
		\begin{eqnarray}\label{1}
			\left\{
			\begin{array}{ll}
				u_{tt}-\nabla\cdot(\sigma\nabla u)+a(x, t)u=0  &\mbox{ in  }Q,\\[2mm]
				u=0 &\mbox{ on }\Sigma,\\[2mm]
				u(x, 0)=u_t(x, 0)=0 &\mbox{ in }\Omega,
			\end{array}
			\right.
		\end{eqnarray}
		where $a(x, t)=\displaystyle\int^1_0 f_u(x, t, su_1(x, t)+(1-s)u_2(x, t))\, ds$.
		By the proof of Theorem 2.1 (see \eqref{AAAA}),  the coefficient $a\in L^\infty(0, T; L^n(\Omega))$. 
		Hence,  by the classical uniqueness result of linear wave equations,  we have 
		directly that $u=0$, that is, $u_1=u_2$ in $Q$.  
	\end{proof}

	\section{Unique determination of  initial data }\label{Sec 4}
	
	In this section, we study the inverse problem  on   determining  initial data for a class of 
	semilinear wave  equations   by the passive   measurement.  As preliminaries,  
	we  first recall an observability result for    the following  wave equation:
	\begin{align}\label{wave2}
		\begin{cases}
			u_{tt}-\Delta u+a(x, t)u=K(x, t)   &\text{ in } Q,\\
			u=0  &\text{ on } \Sigma,\\
			u(x, 0)=\varphi(x), \quad     u_t(x,  0)=\psi(x)   &\text{ in } \Omega,
		\end{cases}
	\end{align}
	where  $a\in  L^\infty(0,  T; L^{p}(\Omega))$ with $p\geq n$, $K\in L^2(Q)$ and  $(\varphi, \psi)\in H^1_0(\Omega)\times  L^2(\Omega)$.
	
	\medskip
	
	Similar to  \cite{8} and \cite{41},  one has  the following  result.
	\begin{lem}\label{lem:Controllability of linear wave equation}
		For any $T>T^*$ $($in  $(\ref{T conditions}))$,  any solution $u\in C([0, T];  H^1_0(
		\Omega))\cap C^1([0, T]; L^2(\Omega))$ to $(\ref{wave2})$  satisfies 
		\begin{align}\label{wave3}
			\begin{split}
				&\|(\varphi, \psi)\|_{H^1_0(\Omega)\times L^2(\Omega)}\\
				\leq & C(T, \Omega, n, \Gamma_0) 
				e^{C\|a\|^{\frac{1}{3/2-n/p}}_{L^\infty(0,  T; L^{p}(\Omega))}}  \Big(
				\norm{\partial_\nu  u}_{L^2(0,  T;  L^2(\Gamma_0))}+\big\|K\big\|_{L^2(Q)}\Big),
			\end{split}
		\end{align}
		where $\Gamma_0$  is given in $(\ref{BBCC})$.
	\end{lem}
	
	\medskip

	Next,  we  give a proof of Theorem \ref{Main Thm 1}.
	
	\begin{proof}[Proof of Theorem $\ref{Main Thm 1}$]
		For any $f\in\mathcal{M}_T$ and    initial values   
		$(\varphi_1, \psi_1),  (\varphi_2,\psi_2)\in H^1_0(\Omega)\times L^2(\Omega)$,  
		denote by $u_1$ and $u_2$,  respectively, the corresponding solutions to (\ref{IBVP for thm 1 for j=1,2}) in $H_0$.  
		Set $\wt u=u_1-u_2.$  Then $\wt u\in H_0$ satisfies 
		\begin{align*}
			\begin{cases}
				\wt u_{tt}-\Delta \wt u+f(x,t,u_1)-f(x, t, u_2)=0  &\text{ in }Q,\\
				\wt u=0  &\text{ on }\Sigma,\\
				\wt u(x, t)=\varphi_1(x)-\varphi_2(x), \quad      \wt u_{t}(x, 0)=\psi_1(x)-\psi_2(x)  &\text{ in } \Omega.
			\end{cases}
		\end{align*}
		By the mean value theorem, 
		\begin{align*}
			&f(x, t, u_1(x, t))-f(x, t, u_2(x, t))\\
			=&\displaystyle\int^1_0 f_u(x, t,  su_1(x, t)+(1-s)u_2(x,  t))ds\cdot \Big[u_1(x, t)-u_2(x, t)\Big].
		\end{align*}
		It follows that
		$$
		\wt u_{tt}-\Delta \wt u+a(x, t)\wt u(x, t)=0,
		$$
		where
		\begin{equation}\label{AAAA}
			a(x,  t)=\displaystyle\int^1_0 f_u(x, t,  su_1(x, t)+(1-s)u_2(x,  t))\, 
			ds\in L^\infty(0, T; L^n(\Omega)).
		\end{equation}
		Indeed,  by (\ref{condition of nonlinear f at infinity data}), for $u_1, u_2\in H_0$ and  any $\epsilon\in (0, 1)$, 
		there exists a positive constant $C_\epsilon$,  such that 
		\begin{align*}
			|a(x, t)|\leq \epsilon\ln\Big(|u_1(x, t)|+|u_2(x, t)|\Big)+C_\epsilon.
		\end{align*}
		By (\ref{global3})  and (\ref{global4})  for $p=n$,  
		\begin{align*}
			e^{\|a\|_{L^\infty(0, T; L^n(\Omega))}}\leq &
			C\Big(1+\sup\limits_{t\in (0, T)}\int_\Omega e^{C|a(x, t)|}\, dx\Big)\\
			\leq & C+C(\epsilon)\sup\limits_{t\in (0, T)}\int_\Omega 
			e^{C\left[\epsilon\ln \LC |u_1(x, t)|+|u_2(x, t)|+1\RC\right]}\, dx\\
			\leq & C(\epsilon)\Big[1+\sup\limits_{t\in (0, T)}\int_\Omega  
			\Big(1+|u_1(x, t)|+|u_2(x, t)|\Big)^{C\epsilon} dx\Big]\\
			\leq & C(\epsilon)\Big[1+\sup\limits_{t\in (0, T)}\int_\Omega  \Big(1+|u_1(x, t)|+|u_2(x, t)|\Big)^2 dx\Big] \\
			\leq &C(\epsilon)\Big(1+\|u_1\|_{C([0, T]; L^2(\Omega))}^2+\|u_2\|_{C([0, T]; L^2(\Omega))}^2\Big),
		\end{align*}
		where $\epsilon>0$ is a sufficiently small constant, such that $C\epsilon\leq 2$, 
		and $C(\epsilon)$ denotes a positive constant, which  depends on $\epsilon$ and 
		may be different from line to line.

		\medskip
		
		By  Lemma \ref{lem:Controllability of linear wave equation},   for  any $p=n$, 
		\begin{align*}
			\left\|(\varphi_1-\varphi_2, \psi_1-\psi_2)\right\|_{H^1_0(\Omega)\times L^2(\Omega)} 
			\leq  C e^{C\|a\|^{2}_{L^\infty(0,  T; L^{n}(\Omega))}} 
			\norm{\partial_\nu \wt u}_{L^2(0,  T;  L^2(\Gamma_0))},
		\end{align*}
		for some  positive constant $C$ depending only on $T$, $\Omega$, $n$ and $\Gamma_0$.
		This implies that the following quantitative stability result:
		\begin{align*}
			&\left\|(\varphi_1-\varphi_2, \psi_1-\psi_2)\right\|_{H^1_0(\Omega)\times L^2(\Omega)}  \\[1mm]
			\leq & C(f, u_1, u_2, n, T, \Omega, \Gamma_0, \Sigma)\cdot\left\|\Lambda_{\varphi_1,\psi_1, f}^0-\Lambda_{\varphi_2, \psi_2, f}^0\right\|_{L^2(0,  T;  L^2(\Gamma_0))},
		\end{align*}
		where $C(f, u_1, u_2,  \Omega, n, T, \Gamma_0)$ is a positive constant depending on $n, 
		T, \Omega, \Gamma_0$, $u_1, u_2$ and $f$, but independent of $(\varphi_j,\psi_j)$, for $j=1,2$.
	\end{proof}  
	\medskip

	On the other hand,  there is a counterexample showing that if  $f$ is unknown,  the passive measurement cannot uniquely determine all unknowns. 
	\begin{thm}[Non-uniqueness]\label{thm:2} Suppose that 
		$f_j\in\mathcal M_T$  and
		$(\varphi_j,\psi_j)\in H^1_0(\Omega)\times L^2(\Omega)$ for $j=1, 2$.  Denote by
		$u_j$  the  solution  to  the following semilinear 
		wave equation:  
		\begin{align}\label{llll}
			\begin{cases}
				u_{j,  tt}-\Delta u_j+f_j(x,t,u_j)=0  &\text{ in } Q,\\
				u_j=0   &\text{ on } \Sigma,\\
				u_j(x,0)=\varphi_j(x), \quad     u_{j, t}(x,0)=\psi_j(x)   &\text{ in } \Omega.
			\end{cases}
		\end{align}
		Then there exist two groups of unknown sources $(\varphi_1, \psi_1, f_1),  (\varphi_2,\psi_2,f_2)
		\in H^1_0(\Omega)\times L^2(\Omega)\times\mathcal M_T$,  such that 
		\[
		(\varphi_1, \psi_1, f_1) \neq (\varphi_2,\psi_2,f_2),
		\]
		but
		\[
		\Lambda^0_{\varphi_1, \psi_1, f_1}=\Lambda^0_{\varphi_2,\psi_2,f_2}.
		\]
	\end{thm}
	
	\begin{proof} Assume that two   functions $u_1, u_2\in C^\infty(\overline{Q})$ satisfy   that  
		\begin{eqnarray*}
			&&u_1(\cdot, 0)\neq u_2(\cdot, 0) \mbox{ in a measurable subset of }\Omega\mbox{  with positive measure}, \\[2mm]
			&&\mbox{and}\  u_1(x, t)=u_2(x, t)=0 \mbox{  in  } \Omega_\epsilon\times[0, T],
		\end{eqnarray*}
		where $\Omega_\epsilon=\Big\{ x\in \Omega\  \Big|\      \dist(x,  \Gamma)<\epsilon  \Big\}$.
		Set 
		$$
		F_j(x, t)=-u_{j, tt}(x,t)+\Delta u_j(x, t), \quad \text{ for }j=1,2 \mbox{ and }(x, t)\in Q.
		$$
		Then  $u_j$   $(j=1, 2)$  are solutions to (\ref{lll}) associated  to 
		\begin{align*}
			&\varphi_j(x)=u_j(x,  0),\  \psi_j(x)=u_{j, t}(x, 0)\mbox{  and  }  f_j(x, t, u_j)=F_j(x,  t).
		\end{align*}
		Notice that
		\[
		(\varphi_1, \psi_1, f_1)\neq (\varphi_2,\psi_2,f_2), 
		\]   
		but $$\partial_\nu u_1\Big|_{\Gamma_0\times(0, T)}
		=\Lambda^0_{\varphi_1, \psi_1, f_1}=\Lambda^0_{\varphi_2,\psi_2,f_2}
		=\partial_\nu u_2\Big|_{\Gamma_0\times(0, T)}=0.$$
		
	\end{proof}

	\medskip

	Finally,  we give  a proof of 
	Corollary \ref{Main Thm 2}.
	
	\begin{proof}[Proof of Corollary $\ref{Main Thm 2}$]
		For  any $f_1, f_2\in \mathcal{C}_T$,  there exists an $f_0\in \mathcal{M}_T$, such that
		$$
		f_1(x, t, s)= f_2(x, t, s)= f_0(x, t, s)\quad\mbox{ in }\Omega\times[0, T^*+\epsilon]\times\mathbb R.
		$$
		Also, by the condition that 
		$\Lambda^0_{\varphi_1, \psi_1, f_1}=\Lambda^0_{\varphi_2, \psi_2, f_2}$, 
		we have that
		$$
		\Lambda^0_{\varphi_1, \psi_1, f_0}=\Lambda^0_{\varphi_2, \psi_2, f_0}  \quad
		\mbox{ on }\Gamma_0\times[0, T^*+\epsilon].
		$$
		Then, by the  results in Theorem \ref{Main Thm 1} for $f=f_0$  and $T=T^*+\epsilon$, 
		the conclusion in Corollary \ref{Main Thm 2} is true.

	\end{proof}

	\section{Simultaneous recovery of initial data and coefficients}\label{Sec 5}
	
	In this section,  the higher order linearization technique  will be used  to determine unknown initial  value and   nonlinear  function in  the  semilinear wave equation (\ref{eq:wave simul})
	simultaneously.  As preliminaries,   based on  the  observability result in Lemma  \ref{lem:Controllability of linear wave equation},  an approximation property for wave 
	equations is given.
	
	\medskip
	
	First, for any $T>T^*$ (see (\ref{T conditions})),   choose two constants $t_1$ and  $t_2$,
	such that 
	$$	T^*<t_1<t_2<T.$$
	Then one has the following approximation result.

	\begin{thm}\label{thm:Runge}
		Assume that      
		$q\in E^{m+1}$ with $\supp\ q\subseteq \overline{\Omega} \times [t_1,t_2]$ for an integer  $m>(n+1)/2$. Then for any solution $v\in C([t_1,t_2];  L^2(\Omega))\cap  C^1([t_1,  t_2];  H^{-1}(\Omega))$ to	\begin{align*}
			v_{tt}-\Delta v +q v =0 \quad\text{ in }Q,
		\end{align*}
		and any $\eps>0$, there exists a solution $V\in C^2(\overline{Q})$ to  
		\begin{align}\label{bigger domain solution}
			\begin{cases}
				V_{tt} - \Delta V +q V=0  &\text{ in }Q, \\
				V(x, 0)=V_t(x, 0)=0 & \text{ in }\Omega,
			\end{cases}
		\end{align}
		such that 
		\[
		\norm{V-v}_{L^2(\Omega \times (t_1,t_2))}<\eps.
		\]
	\end{thm}
	
	\begin{proof}
		In order to prove the desired approximation result, it is equivalent to show that 
		\[
		X=\Big\{  w=V\Big|_{\Omega \times (t_1,t_2)}\  \Big|\  V\in C^2(\overline{Q}) \text{ is a solution to \eqref{bigger domain solution}}  \Big\}
		\]
		is dense in 
		\[
		Y=\Big\{ v\in C([t_1,t_2];  L^2(\Omega))\cap  C^1([t_1,  t_2];  H^{-1}(\Omega)) \  \Big|\  v_{tt} - \Delta v +q v =0 \text{ in }\Omega\times (t_1,t_2) \Big\}
		\]
		in terms of $L^2\LC \Omega \times (t_1,t_2)\RC$. 
		By the Hahn-Banach theorem, it suffices to prove the following statement: if $f \in L^2 (\Omega \times (t_1,t_2))$ satisfies 
		\begin{align}\label{pairing =0 in X}
			\int^{t_2}_{t_1}\int_{\Omega} f w \, dx  dt=0, \quad \forall\  w\in X,
		\end{align}
		it follows that 
		\begin{align}\label{pairing =0 in Y}
			\int^{t_2}_{t_1}\int_{\Omega}  f v \, dx dt =0, \quad \forall\ v\in Y.
		\end{align}

		To this  aim, let $f\in L^2(\Omega \times (t_1,t_2))$ satisfy \eqref{pairing =0 in X} and set
		\begin{align*}
			\wt f(x, t)= \begin{cases}
				f(x, t) & \text{ in }\Omega \times (t_1,t_2), \\
				0 & \text{ in } \Omega \times ((0,t_1]\cup [t_2,T)).
			\end{cases}
		\end{align*}
		Assume that $\wt v\in H_0$ is the solution to the following backward wave equation:  
		\begin{align}\label{equation wt v}
			\begin{cases}
				\wt v_{tt}-\Delta \wt v+q\wt v=\wt f &\text{ in }Q, \\
				\wt v =0 &\text{ on }\Sigma,\\
				\wt v (x, T)=\wt v_t(x, T)=0 &\text{ in } \Omega.
			\end{cases}
		\end{align} 
		Then  for any solution  $V\in C^2(\overline{Q})$ to \eqref{bigger domain solution}  and 
		$w=V\Big|_{\Omega\times(t_1, t_2)}$, one has that 
		\begin{align}\label{pairing =0 on boundary}
			\begin{split}
				0&=\int^{t_2}_{t_1}\int_{\Omega} f w \, dxdt=\int_{Q} \wt f V \, dxdt\\
				&=\int_{Q} \Big(\wt v_{tt}-\Delta\wt v +q \wt v \Big)V \, dxdt=\int_{\Sigma}    \partial_\nu \tilde v\cdot V \, 
				dSdt.
			\end{split}
		\end{align}
		Since $V|_{\Sigma}$ can be  arbitrary function in $
		C^\infty_0(0, T; C^\infty(\Gamma))$,  
		we conclude that   the  associated  solution $V$ to (\ref{bigger domain solution}) satisfies   $V\in E^{m+2}$  and therefore,   $\partial_\nu \wt v =0$ on $\Sigma$. This implies that the solution $\wt v\in H_0$ to (\ref{equation wt v})  satisfies 
		\begin{align*}
			\begin{cases}
				\wt v_{tt}- \Delta\wt v+ q\wt v=\wt f &\text{ in }Q, \\
				\wt v =\partial_\nu \tilde v =0 &\text{ on }\Sigma, \\
				\wt v(x, T)= \wt v_t(x,  T)=0 &\text{ in } \Omega.
			\end{cases}
		\end{align*}
		In particular, in the domain  $\Omega \times ((0,t_1)\cup(t_2,T))$, $\tilde  v\in H_0$ satisfies  the following wave equation: 
		\begin{align*}
			\begin{cases}
				\wt v_{tt} - \Delta \wt v+ q\wt v=0 &\text{ in }\Omega \times ((0,t_1)\cup(t_2,T)), \\
				\wt v =\partial_\nu \tilde v =0 &\text{ on }\Sigma, \\
				\wt v(x,  T)=\wt v_t(x, T)=0 &\text{ in } \Omega.
			\end{cases}
		\end{align*}
		By the observability result in Lemma  
		\ref{lem:Controllability of linear wave equation},   
		we have 
		$$\wt v \equiv 0 \quad\mbox{in  }\ \Omega \times (0,t_1).$$
		By the uniqueness of solutions to wave equations, we  have that 
		$$\wt v \equiv 0 \quad\mbox{in  }\ \Omega \times (t_2, T).$$
		Hence,  
		\begin{align*}
			\begin{cases}
				\wt v(\cdot,  t_1)=\wt v_t(\cdot,  t_1)=\wt v (\cdot, t_2)=\wt v_t(\cdot,  t_2)=0  \text{ in }\Omega,\\[2mm]
				\wt v = \partial_\nu \tilde v =0 \text{ on }\Sigma.
			\end{cases}
		\end{align*}
		It follows that\begin{align*}
			\int^{t_2}_{t_1}\int_{\Omega} f v \, dx dt =\int^{t_2}_{t_1}\int_{\Omega} \Big(\wt v_{tt} - \Delta
			\wt v +q \wt v\Big) v \, dx dt=0,
		\end{align*}
		for any $v\in C([t_1,t_2];  L^2(\Omega))\cap C^1([t_1,t_2];  H^{-1}(\Omega))$ with $v_{tt} -\Delta v + qv=0$ in $\Omega \times (t_2,t_2)$ as desired. This completes the proof of  Theorem \ref{thm:Runge}.
	\end{proof}
	\begin{rmk}
		By the proof of Theorem $\ref{thm:Runge}$,   the approximation property still  holds for the  following more general hyperbolic equation:
		$$
		v_{tt}-\nabla\cdot(\sigma\nabla v)+q v=0 \quad\mbox{in }\Omega\times(t_1, t_2),
		$$
		where $\sigma=\LC \sigma^{ij}(x)\RC_{i,j=1}^n \in C^2(\overline{\Omega}; \R^{n\times n})$ is a symmetric uniformly positive definite matrix-valued function in $\overline\Omega$
		and $(\ref{condition on sigma})$ holds. Also,  $t_1$  and 
		$t_2$ satisfy $T_*<t_1<t_2<T$ for a suitable positive constant $T_*$.
	\end{rmk}
	
	\medskip
	
	Finally, we give a proof of Theorems  \ref{Main Thm:Simultaneous} and \ref{Main Thm:Simultaneous linear}.

	\begin{proof}[Proof of Theorem $\ref{Main Thm:Simultaneous}$]  The  whole  proof is 
		divided into five parts. 
		
		\medskip
		
		{\it Step 1. Initiation} 
		
		\medskip
		
		\noindent For $m>(n+1)/2$, consider the following (lateral) boundary data 
		\begin{align}\label{small lateral BCs}
			h(x,t;\eps)=\sum_{\ell=1}^M \eps_\ell g_\ell\quad \text{on} \   \Sigma, 
		\end{align}
		where $M\in \N$,  $g_1, \cdots,  g_M\in\mathcal{N}_{m+1}$ and $\eps=(\eps_1, \ldots,\eps_M)$ is a parameter vector in $\mathbb R^M$  with   $|\eps|=\sum\limits_{\ell=1}^{M}  |\epsilon_\ell|$  sufficiently small,  such that  
		$$
		\left\|\sum_{\ell=1}^M \eps_\ell g_\ell \right\|_{H^{m+1}(\Sigma)}<\frac{\delta}{2},\  \mbox{   for the }
		\delta>0 \mbox{  in  Theorem  }\ref{thm:well-posedness}.
		$$
		For $j=1,  2$,   let $u_j=u_j(x,t;\eps)\in E^{m+1}$ be solutions to
		\begin{align}\label{IBVP of simultaneous recovery-eps}
			\begin{cases}
				u_{j,  tt}-\Delta u_j+\tilde f_j(x,t,u_j)=0  &\text{ in }Q,\\
				u_j=\displaystyle\sum_{\ell=1}^M \eps_\ell g_\ell\ &\text{ on }\Sigma,\\
				u_j(x, 0)=\varphi_j(x), \quad   u_{j, t}(x, 0)=\psi _j(x) &\text{ in } \Omega,
			\end{cases}
		\end{align}
		where $(\varphi_j, \psi_j)\in H_0^{m+1}(\Omega)
		\times H_0^m(\Omega)$ with $\|(\varphi_j, \psi_j)\|_{H_0^{m+1}(\Omega)
			\times H_0^m(\Omega)}<\delta/2$ and $\tilde f_j$ are admissible coefficients.
		In particular, when $\eps=0 $,  $\wt u_j=u_j(\cdot, \cdot; 0)$ are the solutions to 
		\begin{align}\label{IBVP of simultaneous recovery-eps=0}
			\begin{cases}
				\wt u_{j,  tt}-\Delta \wt u_j +\tilde f_j(x,t, \wt u_j)=0  &\text{ in }Q,\\
				\wt u_j=0   &\text{ on }\Sigma,\\
				\wt u_j(x, 0)=\varphi_j, \quad   \wt u_{j,t}(x, 0)=\psi_j &\text{ in } \Omega.
			\end{cases}
		\end{align}
		We will apply the higher order linearization to the initial-boundary value problem \eqref{IBVP of simultaneous recovery-eps} around the solution $\wt u_j$ to \eqref{IBVP of simultaneous recovery-eps=0} in order to determine  informations on  $\tilde f_j$ for $j=1,2$.
		
		\medskip
		
		{\it Step 2. The first order linearization $(M=1)$}
		
		\medskip
		
		\noindent First,  we linearize the equation \eqref{IBVP of simultaneous recovery-eps} around $\wt u_j$, where $\wt u_j\in  E^{m+1}$ is the solution to \eqref{IBVP of simultaneous recovery-eps=0}, for $j=1,2$. It is easy to show that for $j=1, 2$ and $\ell=M=1$\footnote{In fact, the arguments hold for all $\ell=1,\ldots,M$, where we will use in steps 2-5.},  
		$$
		v_j^{(\ell)}(x,t)=\lim_{\eps_\ell \to 0} \frac{u_j(x,t)-\wt u_j(x,t)}{\eps_\ell}
		$$  satisfies the following wave equation: 
		\begin{align}\label{first linearization}
			\begin{cases}
				v_{j,  tt} ^{(\ell)}-\Delta v_j^{(\ell)}+\tilde q_j v_j ^{(\ell)}=0 &\text{ in }Q, \\
				v_j ^{(\ell)}= g _\ell& \text{ on }\Sigma,\\
				v_j^{(\ell)}(x,  0)= v_{j, t}^{(\ell)}(x, 0)=0 &\text{ in }\Omega,
			\end{cases}
		\end{align}
		where 
		\[
		\tilde q_j (x,t)=\tilde f_{j,  u} (x,t,\wt u_j(x,t))\  \text{ in }Q\quad \mbox{ and }\quad \tilde q_j\in E^{m+1}.
		\]
		It is worth noting that both $\wt u_j$ and $v_j^{(\ell)}$ in \eqref{IBVP of simultaneous recovery-eps=0} and \eqref{first linearization} are still unknown, respectively, since they involve unknown coefficients  and initial data.  In this  step,   we will  show that 
		\begin{align}\label{claim1}
			\tilde f_{1, u} (x,t,\wt u_1(x, t))=\tilde f_{2, u} (x,t, \wt u_2(x, t))\ \text{ in }  \Omega \times (0,T).
		\end{align}
		Recall that we have the same input-output maps	
		$$\Lambda^T_{\varphi_1 , \psi_1, \tilde f_1}(h)=\Lambda^T_{\varphi_2, \psi_2,  \tilde f_2}(h), \quad \text{ for any }h\in  E_\delta.$$
		Hence, with the same boundary data at hand, one has 
		\begin{align}\label{first linearized DN maps agree}
			\begin{array}{rll}
				&v_1^{(\ell)}(x, 0)=v_2^{(\ell)}(x, 0), \quad&v_{1, t}^{(\ell)}(x, 0)=v_{2, t}^{(\ell)}(x, 0), \\[2mm]
				&v_1^{(\ell)}(x, T)=v_2^{(\ell)}(x, T), \quad  &v_{1, t}^{(\ell)} (x, T)= v_{2, t}^{(\ell)}(x, T),\\[2mm]
				&\left. v_1^{(\ell)} \right|_{\Sigma}= \left. v_2^{(\ell)}\right|_{\Sigma}, \quad &	\left. \p _\nu v_1^{(\ell)} \right|_{\Sigma}= \left. \p_\nu  v_2^{(\ell)}\right|_{\Sigma}, 
			\end{array}
		\end{align} 
		for $\ell=M=1$. 
		
		\medskip
		
		Now, subtracting \eqref{first linearization} with $j=1,2$, we have 
		\begin{align}\label{first linearization subtraction}
			\begin{cases}
				v^{(\ell)}_{tt}- \Delta v^{(\ell)} +\tilde q_1 v^{(\ell)}= (\tilde q_2 -\tilde q_1 )v^{(\ell)}_2 & \text{ in }Q, \\
				v^{(\ell)}=0 &\text{ on }\Sigma, \\
				v^{(\ell)}(x,0)=v^{(\ell)}_t(x,0)=0 &\text{ in }\Omega,
			\end{cases}
		\end{align}
		where $v^{(\ell)}:=v^{(\ell)}_1-v^{(\ell)}_2$. Let $\tilde  v_1^{(\ell)}\in 
		C^2(\overline{Q})$ be a solution to
		\begin{align}\label{tilde v_1 equation}
			\tilde  v_{1,tt}^{(\ell)}-\Delta \tilde  v_1^{(\ell)} + \tilde q_1 \tilde  v_1^{(\ell)}=0 \text{ in }Q.
		\end{align}
		Multiplying both sides of  the first equation   in \eqref{first linearization subtraction} by $\tilde  v_1^{(\ell)}$, by  \eqref{first linearized DN maps agree} and an integration by parts,  yields that  
		\begin{align}\label{integral id of 1st linearized equation1}
			\int_{Q} \LC \tilde q_2-\tilde q_1 \RC \tilde v_1^{(\ell)}v_2^{(\ell)}\, dxdt =0.
		\end{align}
		With the admissible conditions in  Definition \ref{Def: admissible coefficients}, \eqref{integral id of 1st linearized equation1} is equivalent to 
		\begin{align}\label{integral id of 1st linearized equation2}
			\int_{t_1}^{t_2}\int_{\Omega} \LC \tilde q_2-\tilde q_1 \RC \tilde v_1^{(\ell)}v_2^{(\ell)}\, dxdt =0.
		\end{align}
		In fact,  by the above arguments, the  equality (\ref{integral id of 1st linearized equation2}) still holds  for complex-valued solutions  $\tilde v_1^{(\ell)}$ and $v_2^{(\ell)}$, respectively,   to (\ref{first linearization})
		with $j=2$ and (\ref{tilde v_1 equation}).
		\medskip
		
		On the other hand,  let $\bm{v}_j\  (j=1, 2)$  be the complex geometrical optics (CGO) solutions  to 
		$$
		\bm{v}_{j, tt}-\Delta \bm{v}_j+ \tilde q_j \bm{v}_j=0 \text{ in }\Omega\times (t_1, t_2)
		$$ in the form in Appendix \ref{Section A}: 
		\begin{align}\label{CGOs in the proof}
			\begin{split}
				\bm{v}_1(x,t)=&e^{-\mbox{i}\tau [\eta(x)+t]}a_1(x,t)+R_1^{(\tau)}(x,t), \\[2mm]
				\bm{v}_2(x,t)=&e^{\mbox{i}\tau [\eta(x)+t]}a_2(x,t)+R_2^{(\tau)}(x,t),
			\end{split}
		\end{align}
		where $\mbox{i}=\sqrt{-1}$ denotes the  imaginary unit,  $\tau\in\mathbb R$ with $|\tau|>1$,  
		$\eta(x)=|x-x_0|$ for an $x_0\in\overline{\Omega}$  and $a_j(\cdot,\cdot)$ has the form \eqref{amplitude function}. 
		Also,  $R_j^{(\tau)}\in L^2(Q)$ is the remainder term, which fulfills the conditions \eqref{zero initial and BC} and \eqref{L2 decay}, for $j=1,2$.

		By  the approximation result (Theorem \ref{thm:Runge}), there are two sequences of complex-valued functions $\left\{ v_{k}^1 \right\}_{k\in \N}$ and  $\left\{ v_{k}^2 \right\}_{k\in \N}$,  such that 
		for $j=1, 2$,  $ v_{k}^j\in C^2(\overline{Q};  \mathbb C)$ is the solution to
		\begin{align}\label{Runge solution1}
			\begin{cases}
				v^j_{k, tt}-\Delta v^j_k+\tilde q_j v^j_{k}=0 &\text{ in }Q, \\[2mm]
				v_{k}^j(x,0)=v^j_{k, t}(x,0)=0 &\text{ in }\Omega,
			\end{cases}
		\end{align} 
		and 
		\begin{align}\label{Runge solution2}
			v_{k}^j\rightarrow \bm{v}_j\quad \text{ in }L^2(\Omega \times (t_1,t_2); \mathbb C), \text{ as }k\to \infty.
		\end{align}
		Choosing $\tilde v_1^{(\ell)}=v_{k}^1$ and $v_2^{(\ell)}=v_{k}^2$ in \eqref{integral id of 1st linearized equation2}, and taking limit, as $k$ tends to $\infty$, one obtains 
		\begin{align}\label{integral id of 1st linearized equation3}
			\int^{t_2}_{t_1}\int_{\Omega} \LC \tilde q_2-\tilde q_1 \RC\bm{v}_1\bm{v}_2\, dxdt =0.
		\end{align}
		It remains to analyze the product $\bm{v}_1\bm{v}_2$ of CGO solutions.
		
		By a direct computation, we have that 
		\begin{align*}
			\begin{split}
				\int^{t_2}_{t_1}\int_{\Omega} \LC \tilde q_2-\tilde q_1 \RC\bm{v}_1\bm{v}_2\, dxdt=\mathbb I +\mathbb I_\tau,
			\end{split}
		\end{align*}
		where 
		\begin{align*}
			\mathbb I =	\int^{t_2}_{t_1}\int_{\Omega} \Big(\tilde q_2-\tilde q_1\Big) a_1 a_2 \, dxdt
		\end{align*}
		and 
		\begin{align*}
			\begin{split}
				\mathbb I_\tau = &\int^{t_2}_{t_1}\int_{\Omega} \LC \tilde q_2-\tilde q_1 \RC \Big\{
				e^{-\mbox{i}\tau [\eta(x)+t]} a_1(x, t)R_2^{(\tau)}(x,t) \\
				& \qquad \qquad + e^{\mbox{i}\tau [\eta(x)+t]} a_2(x, t) R_1^{(\tau)}(x, t) +R_1^{(\tau)}(x,t)R_2^{(\tau)}(x,t)   \Big\} dxdt.
			\end{split}
		\end{align*}
		Since $\tilde q_1, \tilde q_2, a_1$ and $a_2$ are  bounded in $Q$,  and $\norm{R_j^{(\tau)}}_{L^2(\Omega \times (t_1,t_2))}\to 0$, as $|\tau|\to \infty$, for $j=1,2$, it 
		follows that 
		$\mathbb I_\tau\rightarrow 0$, as $\tau\rightarrow \infty$. Hence, 
		the integral identity \eqref{integral id of 1st linearized equation3} implies  
		\begin{align}\label{integral id of 1st linearized equation4}
			\int^{t_2}_{t_1}\int_{\Omega} \LC \tilde q_2-\tilde q_1 \RC a_1(x,t) a_2(x,t) \, dxdt=0.
		\end{align}
		It remains to prove that \eqref{integral id of 1st linearized equation4} implies $\tilde q_1=\tilde q_2$ in $Q$.
		By applying the similar arguments  in \cite[Section 2]{19}, we  conclude that $\tilde q_1 =\tilde q_2 $ in $Q$ as desired.
		Meanwhile,  set
		\begin{align}\label{q(x)}
			q(x,t)=\tilde q_1 (x,t)=\tilde q_2(x,t)  \text{ in } Q.
		\end{align}
		By the uniqueness of solutions to (\ref{first linearization}), one has that 
		\begin{align}\label{uniqueness of solution of the 1st linearization}
			v^{(\ell)}:=v_1^{(\ell)}=v_2^{(\ell)} \quad \text{in}\quad Q,\  \mbox{ for }\ell=1.
		\end{align}

		\medskip
		
		{\it Step 3. The second order linearization  $(M=2)$}
		
		\medskip

		\noindent  For $m=2$,    we  differentiate  (\ref{IBVP of simultaneous recovery-eps}) 
		with respect to different parameters $\eps_1$ and $\eps_2$. 
		It is easy to show that the derivatives $w^{(2)}_j$ $(j=1, 2)$ satisfy
		\begin{align}\label{second linearization}
			\begin{cases}
				w_{j, tt} ^{(2 )}-\Delta w_j^{(2)}+q(x,t) w_j^{(2)} +\tilde f_{j, uu} (x,t,\tilde u_j)v^{(1)}v^{(2)}=0 &\text{ in }Q, \\
				w_j^{(2)} = 0 & \text{ on }\Sigma,\\
				w_j^{(2)}(x, 0)=w_{j, t}^{(2)}(x, 0)=0 &\text{ in }\Omega, 
			\end{cases}
		\end{align}
		where $q, \tilde f_{j, uu}(\cdot, \cdot, \tilde u_j)\in E^{m+1}$ and $v^{(1)}, v^{(2)}\in E^{m+1}$ satisfy 
		\begin{align*}
			\begin{cases}
				v_{tt} ^{(\ell)}-\Delta v^{(\ell)}+q(x,t)v^{(\ell)}=0 &\text{ in }Q, \\
				v^{(\ell)}= g_\ell& \text{ on }\Sigma,\\
				v^{(\ell)}(x,  0)= v_{t}^{(\ell)}(x, 0)=0 &\text{ in }\Omega,
			\end{cases}
		\end{align*}
		here  $g_1,  g_2\in \mathcal{N}_{m+1}$ are arbitrarily given.

		\medskip
		
		Next, we will prove that 
		\begin{align}\label{claim2}
			\tilde f_{1,  uu} (x,t,\wt u_1(x, t))=\tilde f_{2, uu}(x,t,\wt u_2(x, t)) \quad \mbox{in}\quad Q.
		\end{align}
		Similar to the arguments in the first  order linearization, with the equal input-output maps 
		at hand, we  have  that
		\begin{align}\label{second linearized DN maps agree}
			\begin{array}{rll}
				&w_1^{(2)}(x, 0)=w_2^{(2)}(x, 0), \quad&w_{1, t}^{(2)}(x, 0)=w_{2, t}^{(2)}(x, 0), \\[2mm]
				&w_1^{(2)}(x, T)=w_2^{(2)}(x, T), \quad  &w_{1, t}^{(2)} (x, T)= w_{2, t}^{(2)}(x, T),\\[2mm]
				&w_1^{(2)} \Big|_{\Sigma}=w_2^{(2)}\Big|_{\Sigma}, \quad &	\p _\nu w_1^{(2)} \Big|_{\Sigma}=\p_\nu  w_2^{(2)}\Big|_{\Sigma}.
			\end{array}
		\end{align}

		Let $v^{(0)}$ be any solution to
		\begin{eqnarray}\label{AA}
			\left\{
			\begin{array}{ll}
				v_{tt}^{(0)}-\Delta   v^{(0)} + q  v^{(0)}=0 &\text{ in } Q,\\[2mm]
				v^{(0)}(x, 0)=v^{(0)}_t(x, 0)=0 &\text{ in }\Omega.
			\end{array}
			\right.
		\end{eqnarray}
		By subtracting the equations \eqref{second linearization} associated to $j=1,  2$, an integration by parts yields
		\begin{align}\label{second integral id}
			\begin{split}
				&\int_{Q} \Big[\tilde f_{1, uu} (x,t,\wt u_1(x, t))-\tilde f_{2, uu}(x,t,\wt u_2(x, t))\Big] v^{(0)}  v^{(1)}v^{(2)}  \, dxdt \\
				=&\int^{t_2}_{t_1}\int_{\Omega} \Big[\tilde f_{1, uu} (x,t,\wt u_1(x, t))-\tilde f_{2, uu}(x,t,\wt u_2(x, t))\Big]  v^{(0)} v^{(1)}v^{(2)}  \, dxdt =0,
			\end{split}
		\end{align}
		where we have used the admissible conditions for the coefficients.

		As in the first step, we consider the CGO solutions $\bm{v}_1$  and $\bm{v}_2$ of the form \eqref{CGOs in the proof}.
		By using the approximation property again,  we  have that 
		$$ 
		\Big[\tilde f_{1, uu} (x,t,\wt u_1(x, t))-\tilde f_{2, uu}(x,t, \wt u_2(x, t))\Big] v^{(0)}(x, t)=0\   \text{ in } \Omega\times(t_1, t_2).
		$$
		Hence,  for any solution $v_{(0)}$ to (\ref{tilde v_1 equation}), 
		$$
		\int^{t_2}_{t_1}\int_{\Omega} \Big[\tilde f_{1, uu} (x,t,\wt u_1(x, t))-\tilde f_{2, uu}(x,t,\wt u_2(x, t))\Big]  v^{(0)} v_{(0)}  \, dxdt =0.
		$$
		By choosing $v^{(0)}$ and $v_{(0)}$ as the suitable  CGO  solutions again,   we have (\ref{claim2}) as desired. 
		Furthermore, by  the uniqueness of solutions to \eqref{second linearization}, one can immediately obtain 
		\[
		w_1^{(2)}= w_2^{(2)} \quad\text{ in }Q.
		\]

		\medskip
		
		{\it Step 4. The higher order linearization $(M>2)$}
		
		\medskip
		
		\noindent  By the higher order linearization  and the method of induction, we may find $M$-th order derivative of \eqref{IBVP of simultaneous recovery-eps}  and prove that
		\begin{align}\label{claim3}
			\p_u^M \tilde f_1 (x,t,  \wt u_1(x, t))=\p_u^M \tilde f_2 (x,t, \wt u_2(x, t))\quad \text{ in }Q,
		\end{align}
		for any $M=3, 4, \cdots$. To this  aim, we
		first  assume that 
		$$
		\p_u^k \tilde f_1 (x,t, \wt u_1(x, t))=\p_u^k \tilde f_2 (x,t, \wt u_2(x, t))\  \text{ in }Q, \  \text{ for any }k=1,\dots, M-1.
		$$
		Similar to previous steps, 
		by differentiating  \eqref{IBVP of simultaneous recovery-eps} with respect to $\eps_1,\ldots, \eps_{M-1}$ and $\eps_{M}$, one can obtain 
		\begin{align*}
			\int^{t_2}_{t_1}\int_{\Omega}\Big[\p_u ^M \tilde f_1 (x,t, \wt u_1(x, t)) - \p_u^M \tilde f_2 (x,t, \wt u_2(x, t))\Big]  v^{(0)}v^{(1)} \cdots v^{(M)}\, dxdt =0,
		\end{align*}
		where $v^{(\ell)}$ $(\ell=0, 1, \cdots, M)$ are solutions to (\ref{AA}).

		Applying the similar approximation properties   in Step 3,  we  have that
		\begin{align}\label{higher order integral id}	\int^{t_2}_{t_1}\int_{\Omega}\Big[\p_u ^M \tilde f_1 (x,t,u_1^{(0)}(x, t)) - \p_u^M \tilde f_2 (x,t,u_2^{(0)}(x, t))\Big]  v^{(0)}\bm{v}_1\bm{v}_2v^{(3)} \cdots v^{(M)}\, dxdt =0,
		\end{align}
		where  $\bm{v}_1$  and $\bm{v}_2$ are  CGO solutions  of the form \eqref{CGOs in the proof}.
		This implies that
		$$
		\Big[\p_u ^M \tilde f_1 (x,t,u_1^{(0)}(x, t)) - \p_u^M \tilde f_2 (x,t,u_2^{(0)}(x, t))\Big]  v^{(0)}v^{(3)} \cdots v^{(M)}
		=0\quad\mbox{ in }\Omega\times(t_1, t_2).
		$$
		Similar to Step 3, 	if $M$ is odd,  we  take  successively  CGO solutions pairs $v^{(0)}$ and $v^{(3)},$ $\cdots,$ $v^{(M-1)}$ and $v^{(M)}$. Otherwise,  we add a CGO solution to this equality, in order to guarantee even solutions to  be multiplied together.
		It   follows that
		\begin{equation}\label{RRR}
			\p_u ^M \tilde f_1 (x,t, \wt u_1(x, t))=\p_u^M \tilde f_2 (x,t, \wt u_2(x, t))\quad\mbox{ in }Q.
		\end{equation}

		\medskip
		
		{\it Step 5. The determination of  initial data and coefficients}
		
		\medskip
		
		\noindent Recall that  $\wt u_j$ $(j=1, 2)$ are the solutions to the following semilinear wave equation:
		\begin{align*}
			\begin{cases}
				\wt u_{j,  tt}-\Delta \wt u_j +\tilde f_j(x,t, \wt u_j)=0  &\text{ in }Q,\\
				\wt u_j=0   &\text{ on }\Sigma,\\
				\wt u_j(x, 0)=\varphi_j, \  \wt u_{j,t}(x, 0)=\psi_j &\text{ in } \Omega.
			\end{cases}
		\end{align*}
		By the admissible property of $\tilde f_1$ and $\tilde f_2$,
		\begin{eqnarray}\label{LLL}
			\begin{array}{ll}
				&\tilde  f_1(x, t, \wt u_1(x, t))-\tilde  f_2(x, t, \wt u_2(x, t))\\[2mm]
				=&\displaystyle\sum\limits_{k=1}^\infty \frac{\partial_u^k  \tilde f_2(x, t, \wt u_2(x,  t))}{k!}\Big[-\wt u_2(x, t)\Big]^k
				-\sum\limits_{k=1}^\infty \frac{\partial_u^k  \tilde f_1(x, t, \wt u_1(x,  t))}{k!}\Big[-\wt u_1(x, t)\Big]^k\\
				=&\displaystyle\sum\limits_{k=1}^\infty \frac{\partial_u^k  \tilde f_1(x, t, \wt u_1(x,  t))(-1)^k}{k!}\Big\{\Big[\wt u_2(x, t)\Big]^k-\Big[\wt u_1(x, t)\Big]^k\Big\}.
			\end{array}
		\end{eqnarray}
		Since both $\wt u_1$
		and $\wt u_2$ are bounded, set $R=\|\wt u_1\|_{L^\infty(Q)}+\|\wt u_2\|_{L^\infty(Q)}$. Then, for any $L>0$ and $(x, t)\in Q$, 
		\begin{align*}
			&\left|\frac{\tilde  f_1(x, t, \wt u_1(x, t))-\tilde  f_2(x, t,\wt u_2(x, t))}{\wt u_1(x, t)-\wt u_2(x, t)}\right|\\
			=&\left|\sum\limits_{k=1}^\infty \frac{\partial_u^k  \tilde f_1(x, t, \wt u_1 (x,  t))}{k!}(-1)^{k+1}\Big\{\Big[\wt u_1(x, t)\Big]^{k-1}+\Big[\wt u_1(x, t)\Big]^{k-2}\wt u_2(x, t)+\cdots\right.\\
			&\quad \left.+\wt u_1 (x, t)\Big[\wt u_2(x, t)\Big]^{k-2}+\Big[\wt u_2(x, t)\Big]^{k-1}\Big\}\right|\\
			\leq &\sum\limits_{k=1}^\infty \left|\partial_u^k  \tilde f_1(x, t, \wt u_1(x,  t))\right|
			\frac{R^{k-1}}{(k-1)!}  \\
			\leq &\sum\limits_{k=1}^\infty \frac{k R^{k-1} }{L^k} \sup\limits_{|s-\wt u_1(x, t)|=L}  
			|\tilde f_1(x, t, s)|.
		\end{align*}
		Choose  $L=2(R+1)$. By the  admissibility of $\tilde f_1$ and $\tilde f_2$, 
		$$
		G(\cdot, \cdot)=\frac{\tilde  f_1(\cdot, \cdot, \wt u_1(\cdot, \cdot))-\tilde  f_2(\cdot, \cdot, \wt u_2(\cdot, \cdot))}{\wt u_1(\cdot, \cdot)-\wt u_2(\cdot, \cdot)}
		\in L^\infty(Q).
		$$
		Set $w=\wt u_1-\wt u_2$. It is easy to see that
		\begin{align*}
			\begin{cases}
				w_{tt}-\Delta w +Gw=0  &\text{ in }Q,\\
				w=0   &\text{ on }\Sigma,\\
				w(x, 0)=\varphi_1-\varphi_2, \quad w_{t}(x, 0)=\psi_1-\psi_2   &\text{ in } \Omega.
			\end{cases}
		\end{align*}
		By $\Lambda^T_{\varphi_1, \psi_1,  \tilde f_1}(0)=\Lambda^T_{\varphi_2, \psi_2,  \tilde f_2}(0)$
		and the  observability result in  Lemma  \ref{lem:Controllability of linear wave equation}, 
		$$
		\varphi_1=\varphi_2,\quad \psi_1=\psi_2 \quad \mbox{ and }\quad \wt u_1=\wt u_2.
		$$
		By (\ref{LLL}),  
		$$
		\tilde  f_1(x, t, \wt u_1(x, t))=\tilde  f_2(x, t, \wt u_2(x, t)) \quad\mbox{ in }Q.  
		$$
		Furthermore,  notice that for $j=1, 2$ and any $(x, t, s)\in Q\times\mathbb R$,
		\begin{align*}
			&&\tilde f_j(x, t, s)=\tilde f_j(x, t, \wt u_j(x, t))
			+\sum\limits_{k=1}^{\infty} 
			\frac{\partial_u^k \tilde f_j(x, t, \wt u_j(x, t))}{k!}\LC s-\wt u_j(x, t)\RC^k.
		\end{align*}
		With \eqref{RRR} at hand, this implies that $\tilde f_1(x, t, s)=\tilde f_2(x, t, s)$ in  $Q\times\mathbb R$.
	\end{proof}
	
	\medskip
	
	It remains to prove Corollary \ref{Main Thm:Simultaneous linear}.
	\begin{proof}[Proof of Corollary $\ref{Main Thm:Simultaneous linear}$]
		The proof is similar to the one of Theorem \ref{Main Thm:Simultaneous}. 
		For $j=1,2$,  we denote by  $u_j\in E^{m+1}$  the solution to 
		\begin{align*}
			\begin{cases}
				u_{j,tt} - \Delta u_j + q_j u_j =0 &\text{ in }Q,\\
				u_j=h &\text{ on }\Sigma,\\
				u_j(x,0)=\varphi_j(x), \  u_{j,t}(x,0)=\psi_j(x) &\text{ in }\Omega,
			\end{cases}
		\end{align*}
		and denote by $\wt u_j$  the solution to
		\begin{align*}
			\begin{cases}
				\wt	u_{j,tt} - \Delta \wt u_j + q_j \wt u_j =0 &\text{ in }Q,\\
				\wt	u_j=0 &\text{ on }\Sigma,\\
				\wt	u_j(x,0)=\varphi_j(x), \  \wt u_{j,t}(x,0)=\psi_j(x) &\text{ in }\Omega.
			\end{cases}
		\end{align*}
		Let $v_j=u_j-\wt u_j$ and  $v_j\in E^{m+1}$ is the solution to
		\begin{align*}
			\begin{cases}
				v_{j,tt} - \Delta v_j + q_j v_j =0 &\text{ in }Q,\\
				v_j=h &\text{ on }\Sigma,\\
				v_j(x,0)=0, \  v_{j,t}(x,0)=0 &\text{ in }\Omega,
			\end{cases}
		\end{align*}
		for $j=1,2$. 
		By applying the same technique as in the proof of Theorem \ref{Main Thm:Simultaneous}, we first are able to determine $q_1=q_2$ in $Q$. Then,  by  the observability result in Lemma \ref{lem:Controllability of linear wave equation}, we can derive that $\varphi_1=\varphi_2$ and $\psi_1=\psi_2$ in $\Omega$ as desired. This proves the assertion.
	\end{proof}

	\begin{rmk}
		It is worth mentioning that the boundedness of CGO solutions plays an essential role in the proof of Theorem $\ref{Main Thm:Simultaneous}$. With the higher order linearization technique at hand, one can 
		expect the integral identity $(\ref{higher order integral id})$ holds, with products of $M$ solutions of the (linear) wave equation. 
		\begin{itemize}
			\item[(1)] In the elliptic case, one may choose  $v^{(1)}$ and  $v^{(2)}$ as suitable CGO solutions, such that $\{v^{(1)}v^{(2)}\}$ forms a dense subset in $L^1(Q)$. Meanwhile, by applying the maximum principle for the second order linear  elliptic equations, it is not hard to construct bounded positive solutions $v^{(3)},\cdots , v^{(M)}$, such that one can easily derive the global uniqueness result.
			
			\item[(2)] In the hyperbolic case, we do not have the maximum principle and therefore, we do not know the sign and boundedness of certain solutions $v^{(3)},\cdots, v^{(M)}$ in the integral identity $\eqref{higher order integral id}$. Hence, we seek for  the CGO solutions. Indeed, when the CGO solutions are of the form $\eqref{CGOs in the proof}$, they are bounded in $Q$. We refer the readers to Appendix \ref{Section A} for more details about CGO solutions used in our work.
		\end{itemize}
	\end{rmk}

	\noindent $\bullet$ \textbf{Conclusion.}
	
	In this work, we use different measurements to study related inverse problems.
	\begin{itemize}
		\item[(1)] By using the passive measurement, we 
		are able to determine initial data, whenever zeroth order coefficients are  known a priori. On the other hand, the unique determination for initial data cannot hold when nonlinearities are unknown.  However, if the unknown nonlinearities belong to certain classes, one can still determine  the initial data via the passive measurement.

		\item[(2)] By imposing  the admissible  conditions on coefficients, one can recover  initial data and coefficients simultaneously via a hyperbolic type approximation property and the completeness products of
		solutions to wave equations.
		
		\item[(3)] The nonlinearity helps us to study the simultaneous recovery inverse problem. In our approach, when we used the first linearization, the unknown initial data disappears in the first linearized wave equation \eqref{first linearization}. 
		
		\item[(4)] To our best knowledge, for the linear counterpart, Corollary \ref{Main Thm:Simultaneous linear} would be the first result for the simultaneous recovery for both initial data and zero order coefficients. Furthermore, via the proofs of Theorem \ref{Main Thm:Simultaneous} and Corollary \ref{Main Thm:Simultaneous linear}, one can see that the smallness conditions for the semilinear wave equation is needed only for the local well-posedness, but not in the study of inverse problems.
		
	\end{itemize}

	\appendix
	
	\section{Complex geometrical optics (CGO) solutions}\label{Section A}
	
	In this section, let us review the known complex geometrical optics (CGO) solutions for  wave equations with potential. Even though there might be some other references already proved the existence of CGO solutions, we follow the ideas of Kian-Oksanen \cite{19} and give exponential type solutions to a wave equation for the sake of self-containedness of this work.

	Let $\Omega \subseteq \R^n$ be a bounded domain with smooth boundary $\Gamma$, for $n\geq 2$ and $t_1, t_2$ be two positive numbers with $t_1<t_2$\footnote{In the applications, the numbers $t_1,t_2$ are given by Definition \ref{Def: admissible coefficients}.}.
	For any $q\in E^{m+1}$,   consider the wave equation:
	$$
	v_{tt}-\Delta v+qv =0 \quad \mbox{in }\Omega\times (t_1,t_2)
	$$ and its CGO solutions of the  form
	\[
	v(x,t)=a(x,t)e^{\mbox{i}\tau[\eta(x)+t]} + R^{(\tau)} (x,t) \quad\text{ in }\Omega \times (t_1 , t_2), \]
	where $\tau$ is a  real  number  with $|\tau|>1$, 
	$\eta(x)=|x-x_0|$ for an $x_0\in \mathbb R^n\setminus  \overline{\Omega}$, 
	and $R^{(\tau)}$ satisfies
	\begin{align}\label{zero initial and BC}
		\begin{cases}
			R^{(\tau)}_{tt}-\Delta R^{(\tau)}+qR^{(\tau)}=-e^{\mbox{i}\tau(|x-x_0|+t)}\Big(a_{tt}-\Delta a+qa\Big) &\mbox{ in }\Omega\times (t_1, t_2),\\
			R^{(\tau)}=0 & \text{ on }\Gamma\times (t_1,t_2), \\
			R^{(\tau)}(x,  t_1)=R^{(\tau)}_t(x, t_1)=0 & \text{ in }\Omega,
		\end{cases}
	\end{align}
	and 
	\begin{align}\label{L2 decay}
		\lim_{|\tau|\to \infty} \norm{R^{(\tau)}}_{L^2 (\Omega\times(t_1, t_2))}=0.
	\end{align}
	Furthermore, $a(\cdot, \cdot)$ satisfies 
	\begin{align}\label{transport equation}
		2a_{t} - 2 \nabla \eta \cdot \nabla a -\Delta \eta  a =0\quad \text{ in }\Omega\times(t_1, t_2).
	\end{align}

	Without loss of generality, we may assume that $x_0=0$ in the following arguments.
	By  \eqref{transport equation},  $a(\cdot, \cdot)$  satisfies the following equation:
	\begin{align}\label{transport equation1}
		a_t -\frac{x}{|x|} \cdot \nabla a -\frac{n-1}{2|x|}   a=0.
	\end{align}
	As in \cite{19}, we write  $x\in \R^n$ in terms of the polar coordinate $(r,\theta)\in [0, \infty)\times \mathbb S^{n-1}$ and the  metric takes the form $g(r, \theta)=dr^2+g_0(r, \theta)$. Then, \eqref{transport equation1} becomes 
	\begin{align}\label{transport equation2}
		a_t - a_r - \frac{b_r}{4b} \cdot a =0,
	\end{align}
	where 
	\begin{align*}
		b(r,\theta)=\det g_0(r,  \theta).
	\end{align*}
	
	For any $h=h(\theta)\in C^\infty(\mathbb{S}^{n-1})$, $\chi(\cdot)\in C^\infty(\R)$ and $\mu>0$, set 
	\begin{align}\label{amplitude function}
		a(r,\theta,t)=e^{-\frac{\mu (r+t)}{2}} \chi (r+t) h(\theta) b(r,\theta)^{-\frac{1}{4}}.
	\end{align}
	Then $a(\cdot,  \cdot)$  in (\ref{amplitude function}) is the desired solution to the transport equation \eqref{transport equation2}.
	Similarly, 
	$v(x,t)=a(x,t)e^{-\mbox{i}\tau[\eta(x)+t]} + R^{(\tau)} (x,t)$
	is  also the  CGO solution to the wave equation.
	
	\vskip0.5cm

	\noindent\textbf{Acknowledgment.} The work of Y.-H. Lin is partially supported by the Ministry of Science and Technology Taiwan, under the program: 112-2628-M-A49-003. The work of H. Liu is supported by a startup fund from City University of Hong Kong and the Hong Kong RGC General Research Funds (projects 12301420, 12302919 and 12301218).
	The  work of  X. Liu is partially supported  by NSF of China under  grants 11871142 and 11971320.

	\bibliographystyle{alpha}
	
	\bibliography{ref}

\end{document}